%% file: arxiv_dbopt.tex
\documentclass[11pt]{article}
\usepackage[margin=1in]{geometry}

\usepackage{amssymb}
\usepackage{latexsym}
\usepackage{amsmath}
\usepackage{amsthm}
\usepackage{thmtools}
\usepackage{thm-restate}
\usepackage{xcolor}
\usepackage[round]{natbib}
\usepackage{nameref}
\usepackage[colorlinks, linkcolor=blue, filecolor = blue, citecolor = blue, urlcolor = magenta]{hyperref}
\usepackage[noabbrev,capitalize]{cleveref}
\usepackage{graphicx}

\usepackage[extdef=true]{delimset}
\usepackage{enumitem}
\usepackage{algorithm}
\usepackage{algorithmic}
\usepackage{nicefrac}

\input{macros.tex}

\title{\papertitle}

\author{
Aadirupa Saha%
\thanks{Toyota Technological Institute at Chicago (TTIC); {\tt aadirupa@ttic.edu}.}
\and 
Tomer Koren%
\thanks{Blavatnik School of Computer Science, Tel Aviv University and Google Tel Aviv; {\tt tkoren@tauex.tau.ac.il}.} 
\and Yishay Mansour%
\thanks{Blavatnik School of Computer Science, Tel Aviv University and Google Tel Aviv; {\tt mansour.yishay@gmail.com}.}
}
\date{}

\begin{document}

\maketitle

\input{abstract.tex}

\input{introduction.tex}
\input{relatedworks.tex}
\input{problem.tex}

\input{alg_gen.tex}

\input{analysis_smooth2.tex}

\input{alg_strgen.tex}

\input{analysis_strong_a.tex}




\input{conclusion.tex}

\section*{Acknowledgments} 

This project has received funding from the European Research Council (ERC) under the European Union’s Horizon 2020 research and innovation program (grant agreement No.\ 882396), the Israel Science Foundation (grant numbers 993/17; 2549/19), Tel Aviv University Center for AI and Data Science (TAD), the Len Blavatnik and the Blavatnik Family foundation, and the Yandex Initiative for Machine Learning at Tel Aviv University.

\newpage

\bibliographystyle{plainnat}
\bibliography{dueling-refs,opt-refs}

\input{appendix.tex}

\end{document}

%% file: macros.tex
\newcommand{\papertitle}{\makebox[\textwidth]{Dueling Convex Optimization with General Preferences}}

\newcommand{\tk}[1]{\textcolor{magenta}{\bfseries\itshape\{TK: #1\}}}

\newtheorem{thm}{Theorem}
\newtheorem{lem}[thm]{Lemma}

\theoremstyle{definition}
\newtheorem{defn}[thm]{Definition}
\newtheorem{assump}{Assumption}

\newtheorem{rem}{Remark}

\newcommand{\R}{{\mathbb R}}

\newcommand{\N}{{\mathbb N}}
\renewcommand{\P}{{\mathbf P}}

\newcommand{\E}{{\mathbf E}}
\newcommand{\I}{{\mathbf I}}
\newcommand{\1}{{\mathbf 1}}

\newcommand{\cB}{{\mathcal B}}

\newcommand{\cS}{{\mathcal S}}

\newcommand{\cN}{{\mathcal N}}

\newcommand{\cD}{{\mathcal D}}

\newcommand{\cH}{{\mathcal H}}

\newcommand{\tf}{{\tilde {f}}}
\newcommand{\tw}{\tilde \w}

\newcommand{\tc}{\tilde c}
\newcommand{\tkappa}{\tilde \kappa}

\newcommand{\g}{{\mathbf g}}

\newcommand{\n}{{\mathbf n}}

\renewcommand{\u}{{\mathbf u}}
\renewcommand{\v}{{\mathbf v}}
\newcommand{\w}{{\mathbf w}}
\newcommand{\x}{{\mathbf x}}
\newcommand{\y}{{\mathbf y}}
\newcommand{\z}{{\mathbf z}}

\renewcommand{\ref}{\cref}

\newcommand{\dco}{\emph{Dueling Convex Optimization}}

\newcommand{\acgdco}{{G-DCO}}
\newcommand{\pgd}{{$p^{th}$-degree-scaled \emph{Gradient}}}
\newcommand{\kgd}{\textit{Relative-Gradient-Descent}}
\newcommand{\ekgd}{\textit{Epoch-RGD}}
\newcommand{\rhoo}{{\tilde \rho_p}}

\newcommand{\pprox}{$p$-th order proxy}

\newcommand{\nabf}{{\nabla f}}

\DeclareMathOperator{\sign}{\operatorname{sign}}
\newcommand{\dotp}{\boldsymbol\cdot}

\newcommand{\holdon}{D}



%% file: abstract.tex
\begin{abstract}
We address the problem of \emph{convex optimization with dueling feedback}, where the
goal is to minimize a convex function given a weaker form of \emph{dueling} feedback. Each query consists of two points and the dueling feedback returns a (noisy) single-bit binary comparison of the function values of the two queried points.
The translation of the function values to the single comparison bit is through a \emph{transfer function}.
%
%
This problem has been addressed previously for some restricted classes of transfer functions, but here we consider a very general transfer function class which includes all functions that can be approximated by a finite polynomial with a minimal degree $p$.
%
Our main contribution is an efficient algorithm with convergence rate of $\smash{\widetilde O}(\epsilon^{-4p})$ for a smooth convex objective function, and an optimal rate of $\smash{\widetilde O}(\epsilon^{-2p})$ when the objective is smooth and strongly convex.
%
%
\end{abstract}

%% file: introduction.tex
\section{Introduction}
\label{sec:intro}
%

Convex optimization algorithms are fundamental across many fields, including machine learning. 
Most commonly, convex optimization is studied in a first-order gradient oracle model, where the optimization algorithm may query gradients of the objective function.  A more limited model is that of zero-order oracle access, where the optimization algorithm may only query function values of the objective rather than gradients.  Both of these models are extremely well-studied, and the optimal convergence rates in each of them are well known~\cite[see, e.g.,][]{nesterovbook}.

However, there are optimization scenarios where even zero-order access is unavailable or unreliable. Indeed, studies have shown that it is often easier, faster and involves lesser bias to collect feedback on a relative scale rather than asking for reward/loss feedback on an absolute scale. For example to understand the liking for a given pair of items, say (A,B), it is easier for the users to answer preference-based queries like: ``Do you prefer item A over B?", rather than their absolute counterparts: ``How much do you score items A and B in a scale of [0-10]?".
Consequently, relative preference queries are extremely common in 
domains such as recommendation systems, online merchandises, search engine optimization, crowd-sourcing, drug testing, tournament ranking, social surveys, etc \citep{Hajek+14,KhetanOh16}. 
This motivated the introduction of dueling bandits \cite{Yue+09} in the online learning setting.

Drawing motivation from the above, in this paper we study a challenging convex optimization model where the access to the objective function is through a \emph{noisy pairwise comparison oracle}.  
Namely, given an underlying convex objective function $f: \cD \mapsto \R$ ($\cD \subseteq \R^d$ being a convex decision space), at each step the optimization algorithm is allowed to query two points $\w,\w'$ in the feasible domain, upon which only a noisy $1$-bit feedback $o_t \in \{\pm1\}$ is revealed, whose expected value indicates their relative function values.  More specifically, the feedback signal $o_t$ is such that
\begin{align*} \label{eq:pref_model}
    \E[o_t \mid \w,\w'] = \rho( f(\w) - f(\w') ) ,
\end{align*}
where $\rho : \R \mapsto [-1,1]$ is a (possibly nonlinear) \emph{transfer function} mapping difference in function values to a signed preference signal, and $\rho( f(\w) - f(\w') )$ is interpreted as the degree to which $\w$ should be preferred over $\w'$, or vice versa.  Provided such access, our goal is to find a feasible point that approximately minimizes the objective $f$.
Borrowing terminology from the literature on dueling bandits, we call our framework \emph{General Dueling Convex Optimization} (G-DCO) for general transfer functions.

Noisy pairwise comparison access could potentially be significantly weaker than the already weak zero-order access.  Indeed, a special case of this framework has been studied by \citet{Jamieson12} who focused on polynomial transfer functions of the form $\rho(x) = c\sign(x)|x|^p$ and gave tight upper and lower bounds in the pairwise comparison model for \emph{strongly convex and smooth} objectives.  Their results indicate that as $p$ grows larger, the best achievable convergence rate degrades quickly, and already when $p > 1$ this rate becomes strictly inferior to that of zero-order optimization.  Much more recently, \citet{SKM21} considered a similar pairwise comparison model with a different type of a transfer function, namely the sign function $\rho(x) = \sign(x)$ $(p=0)$, and established fast convergence rates for this case exclusively.

Both of these works point us at a some fundamental questions:  Can we design algorithms for dueling convex optimization that are able to leverage more general transfer functions?  Can we converge to a minimizer even when the transfer is unknown to the algorithm?  And what properties of the transfer function dictate the achievable optimization rates?  In this paper, we make progress towards answering these questions.

\subsection{Our contributions}

We make the following main contributions:
\begin{enumerate}[leftmargin=*,label=(\roman*)]
	\item We formalize a generalized dueling convex optimization setting for convex optimization with pairwise-preference feedback given by a general transfer function $\rho$, which is only assumed to be well-behaved around the origin (see \cref{sec:setup} for a precise definition of the query model and optimization objective). Our framework generalizes and significantly extends two existing settings of optimization with comparison feedback~\cite{Jamieson12,SKM21} (\cref{sec:prob}). 
	
	\item We give a novel algorithm for dueling convex optimization with a general transfer function $\rho$, called \kgd\, (\cref{alg:pgd}), which relies on performing a `generalized gradient descent' like update on the \pgd\, of the objective function $f$ (see \cref{def:pgd}, \cref{lem:estpgd}). \cref{rem:estpgd} explains how \pgd\, is a generalization of gradient estimate and smoothly interpolates between different types of descent directions. 
	We prove that when the optimization objective function is smooth, our algorithm needs an order of $O(\epsilon^{-4p})$ queries to the pairwise-preference oracle for finding an $\epsilon$-optimal point (see~\cref{thm:pgd}).  Here, $p$ is the minimal non-zero degree in a series expansion of $\rho(x)$ around zero (\cref{sec:algo}).

	
	\item We further show that our algorithm can achieve faster convergence rates when the function is additionally also strongly convex (\cref{alg:epgd}, \cref{thm:epgd}). Concretely, we show that in this case only $O(\epsilon^{-2p})$ pairwise queries are sufficient for $\epsilon$-convergence.  The latter rate is shown to be tight as it matches existing lower bounds (for certain transfer functions) for strongly convex optimization with comparison feedback due to \cite{Jamieson+15} (\cref{sec:strong}).
\end{enumerate}

Our algorithmic results complement those of \cite{SKM21}, who only considered the sign transfer function.
Compared to the results of \citet{Jamieson12}, we are able to handle both the convex and strongly convex cases (while they only deal with the strongly convex case), and we only require the transfer to be well-behaved around the origin (while they rely on its global structure\footnote{Indeed, their algorithm relies on a line-search procedure at each step, employing the comparison oracle for implementing a one-dimensional noisy binary search.}).  Thus, we are able to encompass a much wider variety of transfer functions whose local behavior around zero is approximated by a polynomial---this includes virtually all functions that admit a series expansion around the origin.



\subsection{Related work}

\paragraph{Dueling Bandits.} 

Due to the widespread applicability and ease of data collection with relative feedback, learning from preferences has gained much popularity in the machine learning community and widely studied as the problem of \emph{Dueling-Bandits} over last decade \citep{Ailon+14,Yue+12,Zoghi+14RUCB,Zoghi+14RCS,Zoghi+15,ADB,Adv_DB,Busa21survey}, which is an online learning framework that generalizes the standard multi-armed bandit (MAB) \citep{Auer+02} setting for identifying a set of `good' arms from a fixed decision-space (set of items) by querying preference feedback of actively chosen item-pairs. 

\paragraph{Limitations of Existing Dueling Bandit techniques.}

Although the relative feedback variants of stochastic MAB problem have been widely studied in the literature, the majority of the existing techniques are restricted to finite decision spaces and stochastic setting which primarily rely on estimating the entries of the underlying preference-matrix . 
These settings, though important as basic steps, are mostly impractical for all real world scenarios which often involves large (or potentially infinite) decision spaces, where lies one of the primary motivation of this work. On the other hand, from an optimization point of view, our work is a key step towards analyzing the fundamental performance limits of function minimization using the weaker form of $0/1$ bit relative preferences. The few existing attempts along this line is discussed in Related Works.

\paragraph{Dueling Bandits in continuous spaces.}
Surprisingly, following the same spirit of extending standard multi-armed bandits (MAB) to continuous decision spaces (as in \emph{linear} or \emph{GP-bandits}), there has not been much work on the continuous extension of the Dueling Bandit problem for large (and structured) decision spaces. 
The works in \cite{Sui+17,PBO} did attempt a similar objective, however, without any satisfactory theoretical performance guarantees. In another recent work, \cite{Brost+16} address the problem of regret minimization in continuous \emph{Dueling Bandits}, however without any finite time regret guarantee of their proposed algorithms.
Recently, \cite{iyengar19,S21} consider the problem of regret minimization from $k$-subsetwise preference feedback  ($k=2$ boils down to the dueling setup) on structured decision spaces, although their underlying utility function is assumed to be only linear, unlike any general convex function considered in our work; moreover, their preference model is restricted only to the class of Multinomial Logit (MNL) based random utility model, unlike the general link function based preference feedback class that we considered.
\cite{CDB,SK21arxiv} represents another line of dueling bandit work, which incorporates context specific dueling preference model. Specifically, their algorithms are designed to compete against an abstract policy set of context to action mappings w.r.t. `minimax-regret'. Their algorithms are also designed to handle potentially large decision spaces, although, the regret objectives are focused to identifying the von-Neumann distribution of the underlying preference models, which is very different from the \emph{function minimization with dueling feedback} point of view that we considered.

\paragraph{Optimization for dueling feedback.}
Along the line of optimization for dueling feedback, \cite{Yue+09} is the first to address the regret minimization problem for fixed functions $f$ (arm rewards) with preference feedback,
although their techniques are majorly restricted to the class of smooth and differentiable preference functions that allows gradient estimation. This is the main reason they could directly apply the classical one-point gradient estimation based \emph{Bandit Gradient Descent (BGD)} algorithm of \cite{Flaxman+04} for the setting, unlike us.
Moreover, another limitation of their framework is their optimization objective is defined in terms of the `preferences' which are directly observable and hence easier to optimize, as opposed to defining it w.r.t. $f$ as considered in this work. 
Following up \cite{Yue+09}, \cite{CntDB} considers the similar problem of dueling bandits on continuous arm set but under rather restrictive sets of assumptions: Twice continuously differentiable, Lipschitz, strongly convex and smooth score/reward function, which are often impractical for modeling any real-world preference feedback.

Closest to our work in spirit are \cite{Jamieson12} and \cite{SKM21}, both of which precisely focus on function optimization with relative pairwise preference feedback.
The latter however is designed to work only under sign based relative feedback which reveals the exact information of which of the two queried points have smaller function value. We instead consider a very general class of polynomial based preference functions (see \cref{sec:prob}) which generalizes the \emph{sign-feedback} model of \cite{SKM21} as a special case. 
While the first, although gives provably optimal convergence rates, their guarantees are restricted to the `well behaved' class of strongly-convex and smooth functions (with bounded Lipschitz gradient). 
The assumptions and consequently their techniques are hence quite restrictive: A major hindrance towards generalizing their algorithmic ideas to a general function class is owning to their line-search based coordinate descent algorithm which is known to fail without strong-convexity.
On the other hand, our algorithm is shown to yield optimal convergence guarantees for more general class of smooth-convex functions. Additionally {we match the convergence rate of} \cite{Jamieson12} with the additional strong convexity assumption which shows the generality of our analysis for a large class of dueling feedback based optimization (\acgdco) problems. 
As motivated in our list of contributions, the novelty lies of our analysis lies in the \kgd\, based optimization approach, which smoothly interpolates between different complexity classes of different \dco\, problems based on the degree of the underlying polynomial link function $p$ (see \cref{rem:estpgd}). Besides our method is arguably simpler both in terms of implementation and analysis.

%% file: problem.tex
\section{Preliminaries and Problem Setup}
\label{sec:prob}

\textbf{Notation.}  Let $[n] = \{1,2, \ldots n\}$, for any $n \in \N$. 
Given a set $S$, for any two items $x,y \in S$, we denote by $x \succ y$ the event $i$ is preferred over $j$. 
For any $r > 0$, let $\cB_d(r)$ and $\cS_d(r)$ denote the ball and the surface of the sphere of radius $r$ in $d$ dimensions respectively.
$\I_d$ denotes the $d \times d$ identity matrix. 
For any vector $\x \in \R^d$, $\|\x\|_2$ denotes the $\ell_2$ norm of vector $\x$. 

\subsection{Problem setup}
\label{sec:setup}
We consider the problem of minimizing a convex and $\beta$-smooth function $f: \cD \mapsto \R$ defined on a bounded convex domain $\cD \subseteq \R^d$ of Euclidean diameter $D$. We denote by $\w^* \in \arg\min_{\w \in \cD} f(\w)$ a point where $f$ is minimized over $\cD$. 


\paragraph{\itshape Query model:}

Our access to the objective $f$ is through a noisy comparison oracle that upon a pair of inputs $(\w,\w') \in \cD^2$ emits a random binary response $o \in \set{\pm 1}$ such that $\E[o \mid \w,\w'] = \rho(f(\w) - f(\w'))$, where $\rho : \mathbb R \to [-1,1]$ is a fixed \emph{transfer function} mapping difference in function values to (signed) preferences, unknown to the algorithm.
For example, given $\rho$ the query model could output a random variable $o$ such that $o \sim \mathrm{Ber}^{\pm}\brk!{\rho(f(\w) - f(\w'))}$ where $\mathrm{Ber}^{\pm}$ denotes a signed version of the Bernoulli distribution (such that for a random variable $X \sim \mathrm{Ber}^{\pm}(p)$, we have $\Pr(X=+1) = 1-\Pr(X=-1) = \frac{p+1}{2}$).

\paragraph{\itshape Transfer function:}

We will assume throughout that the transfer function $\rho: \R \mapsto [-1,1]$ is fixed and unknown to the algorithm.  We make the following assumptions on $\rho$:

\begin{assump}
\label{assump:tf}
(i). $\rho$ is differentiable and anti-symmetric (namely, $\rho(-x) = -\rho(x)$ for all $x$) and satisfies $\rho(0)=0$ and $\sign(\rho(x)) = \sign(x)$ for $x \neq 0$;  
(ii) there are constants $p \geq 1$ and $r,c_\rho>0$ such that for all $x \in (-r,r)$ it holds that $\rho'(x) \geq c_\rho p |x|^{p-1}$. 
\end{assump}

Following gives the intuition behind the practicability of the above set of assumptions: Let us define the function $\rhoo: \R \mapsto [-1,1]$ such that $\rhoo(x)= c_\rho\sign(x)\abs{x}^p$. Then note $\rho$ satisfies ${\rho'(x)} \geq \rhoo'(x) = c_\rho p |x|^{p-1}$ for all $x \in (0,r]$. This essentially implies that our class of admissible transfer functions ($\rho$) admit a series expansion of the form $\rho(x) = \sum_{n=p}^\infty a_n x^n$ around close neighborhood of $x=0$ with minimal degree $p \geq 1$ (see \cref{lem:taylor} for a formal justification). We will henceforth refer \emph{$\rhoo$ as the `\pprox\,' of $\rho$}. Note that, one can recover the `$\sign$\, feedback' of \cite{SKM21} for $\rho = \rhoo$ with $p=0, c_\rho = 1$. \footnote{It is important to note here $\sign$ transfer function is not differentiable, so \cref{assump:tf}-(ii) become vacuous in this case.}

It is also important to note that our assumptions imply that $\rho$ is monotonically increasing in a small neighborhood of the origin.
While we assume that \emph{$\rho$ is unknown} to the algorithm, we will implicitly assume that the parameters $p,r,c_\rho$ above are known. (This knowledge will be used only for optimally tuning the hyper-parameters of our algorithms.)






\paragraph{\itshape Optimization goal:}
The goal of the optimization process is then, given $\epsilon>0$, to find a point $\w$ such that $f(\w)-f(\w^*)\leq \epsilon$ while minimizing the number of queries to to the comparison oracle.

\subsection{Admissible Transfer Functions}

Our latter assumption on the transfer function $\rho$ is perhaps the most stringent one; however, it is satisfied by a wide variety of natural transfer functions: those that admit a series expansion about the origin with uniformly bounded coefficients.
%


\begin{lem} \label{lem:taylor}
Let $\rho$ admit a series expansion $\rho(x) = \sum_{n=p}^\infty a_n x^n$ about $x=0$ with minimal degree $p \geq 1$ and radius of convergence $\delta>0$.
Then, if $a_p > 0$ and $\abs{n a_n} \leq M$ for all $n>p$, we have that
\begin{align*}
    \abs{\rho'(x)} \geq \tfrac12 p a_p \abs{x}^{p-1}
    \quad\text{for}\quad
    \abs{x} < \min\set2{\delta, \frac{p a_p}{4M}}
    .
\end{align*}
\end{lem}

Note that since we require $\rho(0)=0$, it must be that $a_0=0$ and the assumption $p\geq 1$ holds naturally. 
Further, since we would like $\rho(x)>0$ to hold for $x>0$, the first nonzero coefficient must be positive, namely $a_p>0$.
Thus, the only non-trivial assumption is that the series coefficients are uniformly bounded; however, this condition holds for many natural transfer functions: e.g., for the sigmoidal $\arctan(x)$, hyperbolic tangent $\tanh(x)$ and for the error function $\mathrm{erf}(x)$, it holds simply with $M=1$.


\begin{proof}[Proof of \cref{lem:taylor}]
On the interval of convergence $(-\delta,\delta)$ we have $\rho'(x) = \sum_{n=p}^\infty n a_n x^{n-1}$ as one can exchange the order of summation and differentiation. 
Let us write $\rho'(x) = p a_p x^{p-1} + R(x)$, where $R(x) = \sum_{n>p} n a_n x^{n-1}$.
Then, for $\abs{x} < \delta \leq \tfrac12$,
\begin{align*}
    \abs{R(x)} 
    \leq
    \sum_{n>p} \abs{n a_n} \, \abs{x}^{n-1}
    \leq
    M \abs{x}^{p} \sum_{n=0}^\infty \abs{x}^n
    =
    M \abs{x}^{p} \frac{1}{1-\abs{x}}
    \leq
    2M \abs{x}^{p}
    .
\end{align*}
Thus, when $\abs{x} \leq p a_p/4M$ we have $\abs{R(x)} \leq \tfrac12 p a_p \abs{x}^{p-1}$.
It follows that $\abs{\rho'(x)} \geq p a_p \abs{x}^{p-1} - \abs{R(x)} \geq \tfrac12 p a_p \abs{x}^{p-1}$ as claimed.
\end{proof}

%% file: alg_gen.tex
\section{Dueling Convex Optimization with General Transfer Functions}
\label{sec:algo}


In this section we propose an optimization algorithm for our problem (see Objective in \cref{sec:prob}) for any convex and $\beta$-smooth $f$ ($\beta>0$). Note the primary difficulty towards designing an efficient algorithm for the purpose lies in the fact that we can not hope to estimate the gradient of $f$ for any general dueling/pairwise preference model (i.e. any general $\rho$). Thus we can not apply the standard \emph{gradient descent} based techniques to address this problem \citep{boydbook,bubeckbook,hazanbook}.

We get around with the difficulty by noting that, though one may not be able to estimate the exact gradient of $f$, $\nabla f(\w)$, at a given point of interest $\w \in \cD$, we can hope to estimate a `\pprox\, of $\nabla f(\w)$', called \pgd\, of $f$ at $\w$, from the $1$-bit preference feedback generated according to the transfer function (or pairwise preference model) $\rho$. 
The following definition and the lemma describes a more formal argument on this.

\begin{defn}[\pgd]
\label{def:pgd}
Given any function $f:\R^d \mapsto \R$, we define the \pgd\, of $f$ at any point $\w \in \R^d$ to be $\nabla f(\w) \norm{\nabla f(\w)}^{p-1}$ for any $p \ge 1$.
\end{defn}


\cref{lem:estpgd}, in \cref{app:sec3}, gives a formal justification of the key characteristics of \pgd\, estimate. \cref{rem:estpgd} gives a more intuitive explanation of the same and how we exploited it in our optimization algorithm (\cref{alg:pgd}).

\begin{rem}[Key idea behind \cref{alg:pgd}: How it estimates a descent direction in terms of \pgd?] 
\label{rem:estpgd}
As shown in \cref{lem:estpgd} (\cref{app:sec3}), the expected value of our $\g_t$ estimate in \cref{alg:pgd} $(\E_{\u_t,o_t}[\g_t])$, captures the estimated \pgd \, (upto constant factors): It  reflects the direction of the gradient $\nabla f(\w)$ (in expectation) but magnitudewise represents the $p$-order magnitude of that of the true gradient $\norm{\nabla f(\w)}$. Thus $-\g_t$ represents a valid descent direction in expectation, since it points to the negative direction of the gradient (modulo its magnitude is now skewed by the degree $p$).

It is important to note that \pgd\, at any point $\w$ is a power generalization of `gradient feedback' at $\w$, $\nabla f(\w)$, which can automatically smoothly interpolate between different scaling orders of descent directions depending on the `expressiveness' of the transfer function $\rho$ (captured through $p$). Clearly, the best case is attained for $p=1$, when our feedback model is equivalent to the zeroth-order or bandit convex optimization feedback model \citep{Flaxman+04}, when \pgd\, exactly boils down to the gradient  estimate $\nabla f(\w)$. 
Moreover, note if $p=0$, our feedback model recovers the $\sign$-feedback model of \cite{SKM21} and in this case our gradient estimate also $\E[\g_t]$ roughly captures the normalized gradient $\frac{\nabla f(\w)}{\norm{\nabla f(\w)}}$ (direction of the gradient at $\w$) on expectation, as used in \cite{SKM21} as well. 
\end{rem}

\subsection{Algorithm Design: \kgd} 

The crux of the idea lies in designing \pgd\, based algorithm (\cref{alg:pgd}), which is a generalized notion of gradient descent based optimization technique: The algorithm proceeds sequentially, where at each step $t$, it maintains a current point of interest $\w_t \in \cD$, estimate the \pgd\, of $f$ at point $\w_t$ using dueling feedback (as indicated in \cref{lem:estpgd}), and take a `carefully chosen small' step in the negative direction of the estimated \pgd\, to reach the updated point of interest $\w_{t+1}$.  

More formally, the algorithm starts from an initial point $\w_1 \in \cD$. Now at any round $t = 1,2, \ldots$,  the algorithm queries the dueling feedback on a pair of points $(\w_t + \gamma\u_t,\w_t - \gamma\u_t)$, such that $\u_t \sim \text{Unif}(\cS_d(1))$ is any random unit norm $d$-dimensional vector, $\gamma$ being a carefully tuned perturbation parameter.
 Upon receiving the $1$-bit preference feedback $o_t \in \{\pm1\}$, it finds a \pgd\, estimate of $f$ at $\w_t$ as $\g_t := o_t\u_t$ which gives a valid descent direction on expectation as shown in \cref{lem:estpgd} (see \cref{rem:estpgd} for more insights). It then takes an $\eta$-sized step along the negative direction of $\g_t$ to obtain the next iterate $\w_{t+1}:= \w_t - \eta \g_t$ (with suitable projection if necessary). 
The details of the algorithm is presented in \cref{alg:pgd}.
 
\cref{thm:pgd} analyses its convergence guarantees which shows that upon iterating through the above steps for at most $O(\epsilon^{-4p})$ rounds, the algorithm should be able to find a desired $\epsilon$-optimal point.
\vspace{-10pt}
\begin{center}
	\begin{algorithm}[h]
		\caption{\kgd } 
		\label{alg:pgd}
		\begin{algorithmic}[1] 
			\STATE {\bfseries Input:} Initial point: $\w_1 \in \cD$, Learning rate $\eta$, Perturbation parameter $ \gamma$, Query budget $T$ 
			\FOR{$t = 1,2,3,\ldots, T$}
			\STATE Sample $\u_t \sim \text{Unif}(\cS_d(1))$ 
			\STATE Set  $\x_{t}' := \w_t +  \gamma \u_t$,~ 
			            $\y_{t}' := \w_t -  \gamma \u_t$
			\STATE Play the duel $(\x_{t}',\y_t')$, and observe $o_t \in {\pm 1}$ such that $o_t \sim \mathrm{Ber}^{\pm}\big( \rho\big( f(\x_{t}') - f(\y_{t}') \big) \big)$. 
			\STATE Update $\tw_{t+1} \leftarrow \w_t - \eta \g_t$, 
			where $\g_t = o_t \u_t$
			\STATE Project $\w_{t+1} = \arg\min_{\w \in \cD}\norm{\w - \tw_{t+1}}$
			\ENDFOR   
		\end{algorithmic}
	\end{algorithm}
\end{center}
\vspace{-10pt}


Since our proposed \cref{alg:pgd} is based on an iterative `\pgd-descent' based approach (\cref{rem:estpgd}), the interesting part in it's convergence analysis was indeed to understand how this can be exploited to gradually descent towards the true minimizer $\w^*$ and reach an $\epsilon$-optimal point with small enough query complexity. The details are explained more mathematically in the proof of \cref{thm:pgd}.

%% file: analysis_smooth2.tex
\subsection{Convergence Analysis for Smoothly Convex Functions}
\label{sec:analysis_smooth}

\begin{restatable}
{thm}{thmpgd} \label{thm:pgd}
Consider a dueling feedback optimization problem parameterized  by any general admissible transfer function $\rho$ with \pprox\, $\rhoo$ and a $\beta$ smooth convex function $f: \cD \mapsto \R$. 
Then given any $\epsilon > 0$, for the choice of $\gamma = \frac{\tc \epsilon}{\beta \sqrt{d}\holdon}$ and $\eta = \frac{pc_\rho \tc^{2p-1} \epsilon^{2p}}{d^{(2p+1)/2}\beta^p \holdon^{2p-1}}$, there exists at least one $t$ such that $\E[f(\w_{t})] - f(\w^*) \le \epsilon$, after at most $T = \frac{d^{2p+1}\beta^{2p} \holdon^{4p}}{p^2 (\tc^{2p-1} c_\rho)^2\epsilon^{4p}}+1$ iterations;  i.e. $\min_{t \in [T]}\E[f(\w_{t})] - f(\w^*) \le \epsilon$, where $\tc  = \tfrac{1}{20}$ 
is a universal constant.
\end{restatable}

\cref{thm:pgd} shows that for any general transfer function $\rho$ with a $p$-th degree \pprox\, $\rhoo$, \cref{alg:pgd} gives a convergence rate of $O(\epsilon^{-4p})$ to find an $\epsilon$-optimal point. However, \cref{thm:pgdi} shows a improved convergence rate of $O(\epsilon^{-3})$ for linear $(p=1)$ transfer functions which recovers the convergence rate obtained in \cite{sahatewari} for smooth convex functions in the \emph{Bandit Convex Optimization} ($1$ point feedback setting). Moreover, \cref{thm:pgd} also shows how \cref{alg:pgd} can yield a faster convergence rate of $O(\epsilon^{-1})$ for $\sign$ transfer functions $(p=0)$ which matches the convergence rate obtained in \cite{SKM21} --- in fact, not just the final convergence rate, our algorithm (\kgd, \cref{alg:pgd}) generalizes the $\beta$-NGD algorithm of \cite{SKM21} since our descent direction estimate $(\g_t)$, exactly behaves like the normalized gradient (gradient direction) at point $\w_t$ which was the crux of their optimization analysis. Please see the proof of \cref{thm:pgdi} for more details.

\begin{proof}[Proof of \cref{thm:pgd} (sketch)]
The complete details of the proofs can be found in \cref{app:pgd}. 
We denote by $\cH_t$ the history $\{\w_\tau,\u_\tau,o_\tau\}_{\tau = 1}^{t-1} \cup \w_{t}$ till time $t$.
We start by noting that by definition: 
\[
\E_{o_t}[\g_t \mid \cH_t,\u_t] = \rho(f(\w_t+\gamma \u_t) - f(\w_t - \gamma \u_t))\u_t
\]

\textbf{Base Case:} Let us start with the assumption that $f(\w_1)-f(\w^*) > \epsilon$ (as otherwise we already have $\min_{t \in [T]}\E[f(\w_{t})] - f(\w^*) \le \epsilon$ and there is nothing to prove).

We proceed with the proof inductively, i.e. given $\cH_t$ and assuming (conditioning on) $f(\w_t) - f(\w^*) > \epsilon$, we can show that $\w_{t+1}$ always come closer to the minimum $\w^*$ on expectation in terms of the $\ell_2$-norm. More formally, given $\cH_t$ and assuming  $f(\w_t) - f(\w^*) > \epsilon$ we will show:
$
\E_t[\norm{\w_{t+1}-\w^*}^2] \leq \norm{\w_{t}-\w^*}^2,
$ 
where $\E_t[\cdot]:= \E_{o_t,\u_t}[\cdot \mid \cH_t]$ denote the expectation with respect to $\u_t,o_t$ given $\cH_t$. The precise statement can be summarized in the following lemma:

\begin{restatable}[Roundwise Progress of \kgd]{lem}{lempgd}
\label{lem:pgd}
Consider the problem setup of \cref{thm:pgd} and also the choice of $\eta$, $\gamma$. 
Then at any time $t$, during the run of \kgd~(\cref{alg:pgd}), given $\cH_t$, if $f(\w_t) - f(\w^*) > \epsilon$, we can show that:
\begin{align}
\label{eq:fin1orig}
    \E_t[\norm{\w_{t+1}-\w^*}^2] \leq \norm{\w_{t}-\w^*}^2 - \frac{p^2 (\tc^{2p-1} c_\rho)^2\epsilon^{4p}}{d^{2p+1}\beta^{2p} \holdon^{4p-2}}
\end{align}
\end{restatable}

\begin{proof}
We first note that by our update rule,
\begin{align}
\label{eq:ub1orig}
    \E_t[\norm{\w_{t+1}-\w^*}^2] 
    &\leq 
    \E_t[\norm{\w_{t}-\w^*}^2] - 2\eta \E_t[\brk[s]{ \g_t \dotp (\w_t-\w^*) }] + \eta^2 \nonumber
    \\
    &=
    \norm{\w_{t}-\w^*}^2 - 2\eta \E_t[\brk[s]{ \g_t \dotp (\w_t-\w^*) }] + \eta^2 
    .
\end{align}
%


On the other hand, since both $f$ and $\rho$ is convex (by assumption), using \cref{lem:fkm} we get:
\begin{align}
	\label{eq:lb1}
    & \E_t [\g_t \dotp (\w_t-\w^*)] 
    = 
    \E_{\u_t} \brk[s]!{ \rho\brk{ f(\w_t + \gamma \u_t) - f(\w_t - \gamma \u_t) } \u_t \dotp (\w_t-\w^*) \mid \cH_t } \nonumber
    \\
    &= 
    \E_{\u_t} \brk[s]!{ {\rho\brk{ f(\w_t + \gamma \u_t) - f(\w_t - \gamma \u_t) }} \cdot {\u_t } \mid \cH_t } \dotp (\w_t-\w^*) \nonumber
    \\
    &= \nonumber 
    \frac{\gamma}{d}\E_{\u_t}\brk[s]!{ \rho'\big( {f(\w_t+\gamma\u_t) - f(\w_t-\gamma\u_t)} \big) \big( \nabla f(\w_t+\gamma\u_t) + \nabla f(\w_t-\gamma\u_t) \big) \mid \cH_t } \dotp (\w_t-\w^*)
    \\
    &= 
    \frac{\gamma}{d}\E_{\u_t}\brk[s]!{ \rho'\big( \abs{f(\w_t+\gamma\u_t) - f(\w_t-\gamma\u_t)} \big) \big( \nabla f(\w_t+\gamma\u_t) + \nabla f(\w_t-\gamma\u_t) \big) \mid \cH_t } \dotp (\w_t-\w^*)
    ,
\end{align}
%
where the last equality follows since  $\rho(-x) = -\rho(x)$ (see \cref{assump:tf}-i). Now, using convexity of $f$ and $\beta$-smoothness, we can further show that:
\begin{align}
\label{eq:may20}
&	\E_{t} [\g_t \dotp (\w_t-\w^*)] 
    \geq
    \frac{2\gamma}{d} \E_{\u_t}\brk[s]!{ \rho'\big( \abs{ f(\w_t+\gamma\u_t) - f(\w_t-\gamma\u_t) } \big) \brk!{ f(\w_t) - f(\w^*) - \beta \gamma^2 } },
	\\
&  \text{and also } \E_{\u_t}[\abs*{ f(\w_t+ \gamma \u_t) - f(\w_t-  \gamma \u_t)}] \geq \E_{\u_t}[\abs{2 \gamma \u_t \dotp \nabla f(\w_t) }] - \beta \gamma^2
= 2\frac{\tc \gamma \norm{\nabla f(\w_t)} }{\sqrt{d}} - \beta \gamma^2 \nonumber
\end{align}
where the last equality is due to \cref{lem:normgrad}.
Additionally, since by assumption $f(\w_t)-f(\w^*) > \epsilon$, i.e. the suboptimality gap to be at least $\epsilon$, by \cref{lem:gradf_lb} we can further derive a lower bound: 
\begin{align*}
\E_{\u_t}[\abs*{ f(\w_t+ \gamma \u_t) - f(\w_t-  \gamma \u_t)}] \geq \frac{2\tc \gamma \epsilon}{\sqrt{d}\norm{\w_t - \w^*}} - \beta \gamma^2 
= {\frac{2\tc \gamma \epsilon}{\sqrt{d}D}} - \beta \gamma^2. 
\end{align*}
Now for the choice of $\gamma = \frac{\tc \epsilon}{\beta \sqrt{d}\holdon}$ since the lower bound in right hand side is always positive, by monotonicity of $\rhoo'$ in the positive orthant we get: 
\begin{align}
\label{eq:may21}
\rho(\E_{\u_t}[\abs*{ f(\w_t+ \gamma \u_t) - f(\w_t -  \gamma \u_t)}]) & \nonumber
\geq 
 \rhoo'(\E_{\u_t}[\abs*{ f(\w_t+ \gamma \u_t) - f(\w_t -  \gamma \u_t)}]) \\
 & \geq \rhoo'\bigg( {\frac{\tc^2 \epsilon^2}{\beta {d}D^2}} \bigg) 
 = c_\rho p \bigg( {\frac{\tc^2 \epsilon^2}{\beta {d}D^2}} \bigg)^{p-1},
 \end{align}
 where the first inequality follows by the definition of $\rhoo$ which is \pprox\, of $\rho$ (see \cref{assump:tf}), and the last equality follows since by definition $\rho'(x) = c_\rho p x^{p-1} \text{ for any } x \in \R_+$.  
Finally combining \cref{eq:may20,eq:may21}, and the choice of $\gamma$, we can finally derive the lower bound: 
\vspace{-1pt}
\begin{align}
	\label{eq:lb4orig}
	\E_{t} [\g_t \dotp (\w_t-\w^*)] 
	\geq 
	\frac{pc_\rho \tc^{2p-1} \epsilon^{2p}}{d^{(2p+1)/2}\beta^p \holdon^{2p-1}}
	,
\end{align}
Combining \cref{eq:ub1orig} with \cref{eq:lb4orig}:
\begin{align}
    \E_t [\norm{\w_{t+1}-\w^*}^2] 
    &\leq 
    \norm{\w_{t}-\w^*}^2 - 2\eta \E_t[\brk[s]{ \g_t \dotp (\w_t-\w^*) }] + \eta^2 \nonumber
    \\
   &\leq \nonumber 
    \norm{\w_{t}-\w^*}^2
     - 2\eta \Bigg( \frac{pc_\rho \tc^{2p-1} \epsilon^{2p}}{d^{(2p+1)/2}\beta^p \holdon^{2p-1}} \Bigg) + \eta^2
     \\
    &\leq  
    \norm{\w_{t}-\w^*}^2
    - \Bigg( \frac{pc_\rho \tc^{2p-1} \epsilon^{2p}}{d^{(2p+1)/2}\beta^p \holdon^{2p-1}} \Bigg)^2, ~~\text{ setting } \eta = \Bigg( \frac{pc_\rho \tc^{2p-1} \epsilon^{2p}}{d^{(2p+1)/2}\beta^p \holdon^{2p-1}} \Bigg) \nonumber 
    ,
\end{align}
which concludes the claim of \cref{lem:pgd}.
\end{proof}

Now coming back to the main proof of \cref{thm:pgd}, note by iteratively taking expectation over $\cH_T$ on both sides of \cref{eq:fin1orig} and summing over $t=1,\ldots,T$, we get,
\begin{align*}
	\E_{\cH_T} [\norm{\w_{T+1}-\w^*}^2] 
	& \leq 
	\norm{\w_{1}-\w^*}^2
	- \frac{p^2 (\tc^{2p-1} c_\rho)^2\epsilon^{4p}}{d^{2p+1}\beta^{2p} \holdon^{4p-2}}T	 
	.
\end{align*}
However, note if we set $T = \frac{d^{2p+1}\beta^{2p} \holdon^{4p}}{p^2 (\tc^{2p-1} c_\rho)^2\epsilon^{4p}}$,
this implies 
$\E_{\cH_T} [\norm{\w_{T+1}-\w^*}^2] \leq 0$, or equivalently $\w_{T+1} = \w^*$, which concludes the claim. 

To clarify further, note we show that for any run of Alg. \ref{alg:pgd} if indeed $f(\w_{t}) - f(\w^*) > \epsilon$ continues to hold for all $t = 1,2, \ldots T$, then $\w_{T+1} = \w^*$ at $T = \frac{d^{2p+1}\beta^{2p} \holdon^{4p}}{p^2 (\tc^{2p-1} c_\rho)^2\epsilon^{4p}}$. If not, there must have been a time $t \in [T]$ such that $f(\w_t) - f(\w^*) < \epsilon$. 
\end{proof}	


While \cref{thm:pgd} gives the convergence rate of \cref{alg:pgd} for any general `admissible transfer function' $\rho$ (see \cref{sec:prob} for the setup), the following theorem shows that \cref{alg:pgd} can achieve improved convergence rates for certain class of special transfer functions, as remarked in \cref{thm:pgdi}. The proof is given in \cref{app:pgdi}.  

\begin{restatable}[Improved Convergence Rate for Special Transfer Functions.]{thm}{thmpgdi}
\label{thm:pgdi}
\cref{alg:pgd} yields improved $\epsilon$-convergence rate ($T_\epsilon$) for some special class of well-defined transfer functions, e.g.:
\begin{enumerate}[nosep]
\item Linear transfer functions $\rho(x) = c_\rho x, ~\forall x \in \R_+$, then we have $T = \frac{2 d^{2}\beta \holdon^{2}}{c_\rho^2\epsilon^{3}}$; 
\item Sigmoid transfer functions $\rho(x) = \frac{1-e^{-\omega x}}{1+e^{-\omega x}}, ~\forall x \in \R_+$, $\omega > 0$, then we have $T = O\Big(\frac{d^{2}\beta \holdon^{2}}{c_\rho^2\epsilon^{3}}\Big)$. 
\end{enumerate}
\end{restatable}

It is worth noting that, for \emph{Linear transfer functions}, i.e. when $\rho(x)=c_\rho x$ $(p=1)$, our setting is equivalent to the bandit feedback (or zeroth-order) optimization setting \citep{Flaxman+04} and our proposed algorithm obtains the same $O(d^{2/3}T^{-1/3})$ convergence rate of \cite{sahatewari} which is the best known rate till date for zeroth-order smooth convex optimization with gradient descent based algorithms. Moreover, for \emph{Sign transfer functions}, i.e. for $\rho(x)=\sign(x)$ ($p=0$), our algorithm can essentially recovers the $\beta$-NGD algorithm (Algorithm 1) of \cite{SKM21} and hence we can obtain the optimal convergence guarantee of $T = O\Big( \frac{d D \beta}{\epsilon}\Big)$ (see analysis of Case-3 in \cref{app:pgdi} for details).
These results thus show the generalizability of our problem framework as well as our algorithmic approach (\cref{alg:pgd}).

%% file: alg_strgen.tex
\section{Strongly Convex Dueling Optimization}
\label{sec:strong}

In this section, we analyze an epoch-wise version of \kgd\, (\cref{alg:pgd}) which is shown to yield better convergence guarantees for $\alpha$-strongly convex and $\beta$-smooth functions. 
The key idea lies in noting that in order to design an optimal algorithm for $\alpha$-strongly convex $\beta$-smooth functions, one can simply iteratively reuse any $\beta$-smoothly convex optimization routine (e.g. we can use our Alg. \ref{alg:pgd}) by running it as a black-box over a successive number of epoch-wise warm-starts.  
Our resulting convergence analysis (\cref{thm:epgd}) shows that, in this case the algorithm can find an $\epsilon$-optimal point upon querying just $O(\epsilon^{-2p})$ pairwise comparisons (as opposed to the $O(\epsilon^{-4p})$ sample complexity rate for the $\beta$-smooth case, see \cref{thm:pgd}).
This is possible due to the  nice properties of strong convexity, where nearness in the suboptimality gap in terms of the function values, $f(\w)-f(\w^*)$ implies nearness in terms of the $\ell_2$-distance from $\norm{\w-\w^*}$ (see the third property in Lem. \ref{lem:prop_a}). 
In fact the $O(\epsilon^{-2p})$ convergence rate can shown to be information theoretically optimal (see \cref{rem:epgd}).

\subsection{Algorithm Design: \ekgd} 

As motivated above, our proposed method \ekgd \, (\cref{alg:epgd}) uses an `epoch-wise black-boxing of a smooth-convex optimization routine' with `warm-starting' approach. For our purpose, we use the earlier proposed \kgd\, (\cref{alg:pgd}) as the black-box. More formally, the algorithm, starts with some initial point $\w_1$ and runs over a sequence of $k_\epsilon = O\Big( \log \frac{\beta D^2}{\epsilon} \Big)$ epochs: Inside each epoch $k \in [k_\epsilon]$, we call the \kgd$(\w_k,\eta_k,\gamma_k,t_k)$ subroutine with the the initial (warm-start) iterate $\w_k$,   suitably tuned parameters $\eta_k,\gamma_k$ and a query budget of $t_k$. The decision point returned by \kgd\, after $t_k$ steps is considered to the next iterate, setting $\w_{k+1} \leftarrow $ \kgd$(\w_k,\eta_k,\gamma_k,t_k)$ and we proceed to the $(k+1)$-th epoch, warm-starting it with $\w_{k+1}$.

The key idea behind the epoch-wise warm-start approach exploits the fact that between any two consecutive epochs, say $k$ and $k+1$, the $\ell_2$ distance of $\w_k$ from $\w^*$ gets reduced by a constant fraction on an expectation (\cref{lem:epgd}). Thus, it can be shown that running the algorithms for roughly $k_{\epsilon} = O\Big( \log \frac{\beta \norm{\w_1-\w^*}^2}{\epsilon} \Big)$ epoch, would lead to $\norm{\w_{k_\epsilon}-\w^*}^2 \leq \epsilon$, which in turn imply the $\epsilon$-convergence (see the proof of \cref{thm:epgd} for details). The formal description of the algorithm is given in \cref{alg:epgd}.

\vspace{-5pt}
\begin{center}
	\begin{algorithm}[h]
		\caption{\bf \ekgd($\epsilon$)}
		\label{alg:epgd}
		\begin{algorithmic}[1] 
			\STATE {\bfseries Input:} Error tolerance $\epsilon > 0$%
			\STATE {\bf Initialize} Initial point: $\w_1 \in \R^d$ such that $\norm{\w_1-\w^*} \leq D$, Phase count $k_\epsilon:= \lceil \log_{4/3} \frac{\beta D^2}{2\epsilon}  \rceil$ 
			\\ $D_1 = D$, $B:= \frac{2c_\rho p}{(\alpha+\beta)}\Big( \big(\nicefrac{\alpha^2}{4\beta} \big)^{p} \frac{\tc^{2p-1} }{d^{\frac{2p+1}{2}}} \Big)$, $\tc \leftarrow$ the universal constant from \cref{lem:normgrad}.
			%
			\FOR{$k = 1,2,3,\ldots, k_\epsilon$}
			\STATE $\eta_k \leftarrow B D_k^{2p+1}$, $ \gamma_k \leftarrow \frac{\tc \alpha D_k}{2\beta \sqrt{d}}$, $t_k = \frac{1}{2B^2{(D_k^{2})}^{2p}}$, $D_{k+1} \leftarrow \sqrt{\frac{3}{4}}D_{k}$.
			\STATE Update $\w_{k+1} \leftarrow$ \kgd$\big(\w_{k},\eta_k, \gamma_k,t_k\big)$
			\ENDFOR   
			\STATE Return $\w_{k_\epsilon+1}$
		\end{algorithmic}
	\end{algorithm}
\end{center}
\vspace{-5pt}

%% file: analysis_strong_a.tex
\subsection{Convergence Analysis for Smooth and Strongly Convex Functions}
\label{sec:analysis_strong}

\begin{restatable}[Convergence Analysis of \ekgd\, for Smooth and Strongly convex Functions]{thm}{thmepgd}
\label{thm:epgd}
Consider a dueling feedback optimization problem parameterized  by any general admissible transfer function $\rho$ with \pprox\, $\rhoo$ and a $\beta$ smooth $\alpha$-strongly convex function $f: \cD \mapsto \R$. 
Then given any $\epsilon > 0$, the the final point $\w_{k_\epsilon+1}$ returned by \cref{alg:epgd} satisfies $\E[f(\w_{t+1})] - f(\w^*) \le \epsilon$, with a sample complexity of at most $O(\frac{1}{B^2{\epsilon}^{2p}})$ pairwise comparisons. (Here the constant $B$ is as defined in \cref{alg:epgd}, $\tc =  \tfrac{1}{20}$ is a universal constant. 
\end{restatable}

 \begin{rem}[Optimal Convergence of \ekgd\, for Strongly Convex and Smooth Functions]
 \label{rem:epgd}
Note if $\rho(\x)$ is exactly of the form $\rho(\x) = \sign(\x)\abs{x}^p$, then our dueling (pairwise preference) feedback model is equivalent to the same used in \cite{Jamieson12}. It is interesting to note that, their derived lower bound sample complexity for the $\epsilon$-convergence for smooth and strongly convex functions was indeed shown to be $\Omega(\epsilon^{-2p})$ which implies the 
optimality of \ekgd\,(\cref{alg:epgd}) for $\alpha$-strongly convex and $\beta$-smooth functions for any values of $p$. 
The line search algorithm proposed by \cite{Jamieson12} also achieves the same convergence rate for strongly convex functions, modulo some additional multiplicative polylogarithmic terms in $d, \epsilon$ etc, which we do not incur. Also \ekgd\, is much more modular and simpler to implement as well as relatively easier to analyze.
%
Besides, the application scope of \ekgd\, is much more general that applies to the class of any general transfer function (as discussed in \cref{sec:prob}) and also works for non-strongly convex functions (see \cref{thm:pgd}). 

Moreover, \cref{thm:epgd} shows that \cref{alg:epgd} actually gives optimal rates for the special transfer functions studied earlier, e.g., $O(\epsilon^{-2})$ convergence rate for \emph{linear transfer function} $(p=1)$ a.k.a. zeroth-order feedback model (as proved in \cite{Hazan14}), or sigmoid based preference feedback (see \cref{thm:pgdi}). Besides it also yields the optimal $O(\log \frac{1}{\epsilon})$ convergence rate for $\sign$ feedback $(p=0)$ (see Theorem $7$ in \cite{SKM21}), since note our algorithm is essentially a generalization of Algrithm 2 of \cite{SKM21} which we can easily recover with the proper tuning of the algorithm parameters $\eta_k,\gamma_k,t_k$ ($k \in [k_\epsilon]$). 
\end{rem}

\begin{proof}[Proof of \cref{thm:epgd} (sketch)]
The complete details of the proofs can be found in \cref{app:epgd}. 
The proof of the main theorem is based on a key lemma that shows after every epoch of length $t_k$, the distance of the resulting point $\w_{k+1}$ from the optimal $\w^*$ must decrease by at least a constant fraction. The formal statement is given below:

\begin{restatable}[Epochwise Convergence Guarantee of \ekgd]{lem}{lemepgd}
\label{lem:epgd}
Consider the problem setup of \cref{thm:epgd}. 
Then the point $\w_{k+1}$ returned by $k$-th epoch run of \ekgd~(\cref{alg:epgd}) starting form the initial point $\w_k$, satisfies: 
\[
\E[\norm{\w_{k+1}-\w^*}^2 \mid \w_k] \leq \frac{3}{4} \norm{\w_{k}-\w^*}^2,
\]
$\forall k \in [k_\epsilon]$. Where the expectation is taken over the randomness of the algorithm and the dueling feedback received inside the run of \kgd$\big(\w_{k},\eta_k, \gamma_k,t_k\big)$.
\end{restatable}
The main part of the proof of \cref{lem:epgd} follows along the similar line of argument as of \cref{lem:pgd}, however we need to carefully apply the properties of $\alpha$ strong-convexity of $f$ in order to achieve the improved convergence rates. The complete details can be found \cref{app:lemepgd}. Given \cref{lem:epgd}, claim of \cref{thm:epgd} now follows from the following \textbf{epoch-wise recursion argument: }

Let $\cH_{[k]}:= \{(\w_{k'})_{k'\in [k_\epsilon]},(\u_{t'},o_{t'})_{t' \in [\sum_{k'= 1}^{k}t_{k'}]}\} \cup \{\w_{k+1}\}$ denotes the complete history till the end of epoch $k$ starting from the first epoch $\forall k \in [k_\epsilon]$. 

Further, let us denote by $\cH_{k}:= \{\w_{k},(\u_{t'},o_{t'})_{t' \in [\sum_{k'= 1}^{k-1}t_{k'}+1,\sum_{k'= 1}^{k}t_{k'}]}\} \cup \{\w_{k+1}\}$ be the history only within epoch $k$. 

\textbf{Proof of Correctness. } From \cref{lem:epgd}, note we have already established 

\[\E_{\cH_k}[\norm{\w_{k+1}-\w^*}^2 \mid \w_k] \leq {\frac{3}{4}}\norm{\w_{k}-\w^*}^2.
\]

Applying the argument iteratively over $E$ epochs, and the law of iterated expectations, we have:
\begin{align}
\label{eq:may19aorig}
\E_{\cH_{[E]}}[\norm{\w_{E+1}-\w^*}^2 ] \leq (\nicefrac{3}{4})^E\norm{\w_{1}-\w^*}^2.
\end{align}

Thus choosing $E = \lceil \log_{4/3} (\frac{\beta D^2}{2\epsilon}) \rceil$, where $\norm{\w_1 - \w^*} \leq D$,
we have 
$
E \geq \log_{4/3} (\frac{\beta \norm{\w_1 - \w^*}^2}{2\epsilon})
	~
	\implies 
	~
	(\nicefrac{3}{4})^{E}\norm{\w_1 - \w^*}^2 \leq \frac{2\epsilon}{\beta}.
$

Thus from \eqref{eq:may19aorig}, we get: 
$
\E_{\cH_{[E]}}[\norm{\w_{E+1}-\w^*}^2 ] \leq \frac{2\epsilon}{\beta},
$
and further applying $\beta$-smoothness of $f$, we get:
\begin{align*}
\E_{\cH_{[E]}}[f(\w_{E+1})-f(\w^*) ] \leq \E_{\cH_{[E]}}[\frac{\beta}{2}\norm{\w_{E+1}-\w^*}^2 ] \leq \epsilon,
\end{align*}
which proves the correctness of \cref{alg:epgd} for the choice of total number of epochs $k_{\epsilon}=E$.

\textbf{Proof of Sample Complexity. } In order to verify that \cref{alg:epgd} indeed converges to an $\epsilon$-optimal point in $O(\epsilon^{-2p})$ sample complexity, note we simply need to count the total sample complexity incurred in the $k_\epsilon$ epochwise runs of \kgd\, (see Line \#5 of \cref{alg:epgd}). However, by design of \ekgd\,(\cref{alg:epgd}), since \kgd$\big(\w_{k},\eta_k, \gamma_k,t_k\big)$ is run for only $t_k = \frac{1}{2B^2{(D_k^{2})}^{2p}}$ iterations, the total sample complexity of \ekgd\, becomes:

\begin{align*}
\sum_{k = 1}^{k_\epsilon} t_k 
	&= \frac{1}{2B^2}\sum_{k = 1}^{k_\epsilon} \frac{1}{{(D_k^{2})}^{2p}} 
	= \frac{1}{2B^2 (D^2)^{2p}}\bigg( 1 + \frac{1}{(\nicefrac{3}{4})^{2p}} + \frac{1}{((\nicefrac{3}{4})^{2p})^2} + \cdots + \frac{1}{((\nicefrac{3}{4})^{2p})^{k_\epsilon-1}} \bigg)
	\\
	& = \frac{1}{2B^2 (D^2)^{2p}}\frac{(\nicefrac{4}{3}^{2p})^{k_\epsilon} -1}{\nicefrac{4}{3}^{2p}-1} \leq  \frac{1}{4B^2 (D^2)^{2p}}\Big( (\nicefrac{\beta D^2}{2 \epsilon})^{2p} -1 \Big) = O\Big(\frac{1}{B^2\epsilon^{2p}}\Big)
\end{align*}
where the last inequality is since $k_\epsilon = \lceil \log_{4/3} (\frac{\beta D^2}{2\epsilon}) \rceil$ by definition. 
Thus follows the claimed sample complexity of \ekgd\, in \cref{thm:epgd} and this concludes the proof.
\end{proof}

%% file: conclusion.tex
\vspace{-7pt}
\section{Conclusion and Perspective}
\label{sec:concl}
\vspace{-6pt}
We consider the problem of convex optimization under a general class of pairwise preferences (dueling) feedback. 
Note the primary difficulty towards designing an efficient algorithm for the purpose lies in the fact that we can not hope to estimate the gradient of $f$ for any general dueling/pairwise preference model. Thus we can not apply the standard \emph{gradient descent} based techniques to address this problem.
We get around with the difficulty by estimating a $p$-th order proxy of the gradient, called \pgd. The crux of the idea lies in designing \kgd\, based algorithm (\cref{alg:pgd}), which is a generalized notion of gradient descent based optimization technique. Using this we design an efficient algorithm with convergence rate of $\smash{\widetilde O}(\epsilon^{-4p})$ for a smooth convex objective function, and an optimal rate of $\smash{\widetilde O}(\epsilon^{-2p})$ when the objective is smooth and strongly convex.

\paragraph{Future work.} Although the derived convergence rate for the strongly convex setting is information theoretically tight, the exact convergence lower bound is unclear for the class of smooth functions, which might be an interesting problem to pursue independently. Another open problem is to analyze this problem beyond the smoothness assumption.
Considering  a regret minimization objective instead of the optimization perspective, as well as understanding the information theoretic regret performance limit would be interesting direction as well. 
One can also consider generalizing the optimization framework to subsetwise preferences, instead of just pairwise (dueling) feedback. It might also be useful to extend our setup for contextual scenarios, adversarial preferences or non-stationary function sequences and understand the scopes of feasible solutions as well as the impossibility results.

%% file: appendix.tex
\appendix

\onecolumn{
	
\section*{\centering\Large{Supplementary: \papertitle}}
	
\allowdisplaybreaks
	
\section{Useful Results (used in \cref{sec:algo,sec:analysis_smooth,sec:analysis_strong})}
\label{app:link2grad}	
	
\begin{restatable}[]{lem}{lemngrad}
\label{lem:normgrad}	
For a given vector $\g \in \R^d$ and a random unit vector $\u$ drawn uniformly from $\cS_d(1)$, we have
\begin{align*}
    \E_{\u}[\abs{\g \dotp \u} ]
    =
    \frac{\tc {\norm{\g}}}{\sqrt{d}} 
    ,
\end{align*}
for some universal constant $\tc \in [\tfrac{1}{20},1]$. 
\end{restatable}	
	
\begin{proof}
Without loss of generality we can assume $\norm{\g} = 1$, since one can divide by $\norm{\g}$ in both side of Lem. \ref{lem:normgrad} without affecting the claim.
Now to bound $\E[\abs{\g \dotp \u}]$, note that since $\u$ is drawn uniformly from $\cS_d(1)$, by rotation invariance this equals $\E[\abs{u_1}]$.
For an upper bound, observe that by symmetry $\E[u_1^2] = \tfrac1d \E[\sum_{i=1}^d u_i^2] = \tfrac1d$ and thus
\begin{align*}
    \E[\abs{u_1}]
    \leq
    \sqrt{\E[u_1^2]}
    =
    \frac{1}{\sqrt{d}}
    .
\end{align*}
We turn to prove a lower bound on $\E[\abs{\g \dotp \u}]$.
If $\u$ were a Gaussian random vector with i.i.d.~entries $u_i \sim \cN(0,1/d)$, then from standard properties of the (truncated) Gaussian distribution we would have gotten that $\E[\abs{u_1}] = \sqrt{2/\pi d}$. 
For $\u$ uniformly distributed on the unit sphere, $u_i$ is distributed as $v_1/\norm{\v}$ where $\v$ is Gaussian with i.i.d.~entries $\cN(0,1/d)$.
We then can write
\begin{align*}
    \Pr\brk*{ \abs{u_1} \geq \frac{\epsilon}{\sqrt{d}} }
    & =
    \Pr\brk*{ \frac{\abs{v_1}}{\norm{\v}} \geq \frac{\epsilon}{\sqrt{d}} }
    \geq
    \Pr\brk*{ \abs{v_1} \geq \frac{1}{\sqrt{d}} \mbox{ and } \norm{\v} \leq \frac{1}{\epsilon} }
    \\
    & \geq
    1 - \Pr\brk*{ \abs{v_1} < \frac{1}{\sqrt{d}}} - \Pr\brk*{\norm{\v} > \frac{1}{\epsilon} }
    .
\end{align*}
Since $\sqrt{d} v_1$ is a standard Normal, we have
\begin{align*}
    \Pr\brk*{ \abs{v_1} < \frac{1}{\sqrt{d}}}
    =
    \Pr\brk*{ -1 < \sqrt{d} v_1 < 1}
    =
    2\Phi(1) - 1
    \leq
    0.7
    ,
\end{align*}
and since $\E[\norm{\v}^2] = 1$ an application of Markov's inequality gives
\begin{align*}
    \Pr\brk!{\norm{\v} > \frac{1}{\epsilon}}
    =
    \Pr\brk!{\norm{\v}^2 > \frac{1}{\epsilon^2}}
    \leq
    \epsilon^2 \E[\norm{\v}^2]
    =
    \epsilon^2
    .
\end{align*}
For $\epsilon=\tfrac14$ this implies that $\Pr\brk!{ \abs{u_1} \geq 1/4\sqrt{d} } \geq \tfrac15$, whence
$
    \E[\abs{\g \dotp \u}]
    =
    \E[\abs{u_1}]
    \geq
    1/20\sqrt{d}
    .
$
\end{proof}
	
\begin{lem}
	\label{lem:fkm}
Let $g : \R^d \to \R$ be differentiable and let $\u$ be a random unit vector in $\R^d$.
Then
\begin{align*}
    \E[g(\u)\u] 
    = 
    \frac1d \E[\nabla g(u)]
    .
\end{align*}
\end{lem}

\begin{proof}
The above claim follows from Lemma 1 of \cite{Flaxman+04} which shows that for any differentiable function $f : \R^d \to \R$ and any $\x \in \R^d$,
\begin{align*}
    \E[ f(\x+\gamma \u) \u ]
    =
    \frac{\delta}{d} \nabla \E[ f(\x+\gamma \u) ]
    =
    \frac{\delta}{d} \E[ \nabla f(\x+\gamma \u) ]
    .
\end{align*}
Fix $\x=0$ and substitute $f(\z) = g(\tfrac1\gamma \z)$.
Then $\nabla f(\z) = \frac1\gamma \nabla g(\frac1\gamma \z)$, and we obtain:
\begin{align*}
    \E[ g(\u) \u ]
    =
    \frac{1}{d} \E[ \nabla g(\u) ]
    .
\end{align*}
\end{proof}

\begin{restatable}[]{lem}{gradflb}
	\label{lem:gradf_lb}
	Suppose $f: \cD \mapsto \R$ is a convex function for some convex set $\cD \subseteq \R^d$ such that for any $\x,\y \in \R^d$, $f(\x) - f(\y) > \epsilon$. Then this implies $\| \nabla f(\x)\| > \frac{\epsilon}{\|\x - \y\|}$. 
	Further, assuming $D: = \max_{\x, \y \in \R^d}\|\x-\y\|_2$, we get $\| \nabla f(\x)\| > \frac{\epsilon}{D}$ for any $\x \in \cD$.
\end{restatable}

\begin{proof}
	The proof simply follows using convexity of $f$ as:
	\begin{align*}
		f(\x) - f(\y) > \epsilon & \implies \epsilon < f(\x) - f(\y) \le \nabla f(\x)(\x - \y) \le \| \nabla f(\x)\|_2 \|\x - \y\|_2 \\
		& \implies \| \nabla f(\x)\| \ge \frac{\epsilon}{\|\x - \y\|}.
	\end{align*}
	\vspace{-20pt}
\end{proof}


\section{Appendix for \cref{sec:algo}}
\label{app:sec3}

\begin{restatable}[Estimation of \pgd\, from Dueling Feedback from general transfer function $\rho$]{lem}{estpgd}
	\label{lem:estpgd}
	Consider any general admissible transfer function $\rho: \R  \mapsto [-1,1]$ (that satisfies \cref{assump:tf}). 
	Then if $f: \cD \mapsto \R$ is $L$-Lipschitz and $\beta$-smooth, given any $\w \in \cD$ such that $r \geq \frac{2L \norm{\nabla f(\w)}}{\beta \sqrt d}$, $o \sim \mathrm{Ber}^{\pm}\big( \rho\big( f(\w + \gamma\u) - f(\w - \gamma\u) \big) \big)$ where $\u \sim \text{Unif}(\cS_{d}(1))$, and suitably tuned step-size $\gamma > 0$:
	\begin{align*}
		\E_{\u,o}&[o \u]\dotp(\w-\w^*) 
		\geq 
		\begin{cases}
		C(d,\beta,c_\rho,p,r,L) \norm{\nabla f(\w)}^{2p-1} \nabla \tf(\w) \dotp (\w-\w^*), 
		&\text{ for any } p \geq 1 ,
		\\
		 \frac{4\gamma c_\rho}{d}\nabla \tf(\w) \dotp (\w-\w^*), 
		 &\text{ for } p = 1.
		\end{cases}
	\end{align*}
Here $\nabla \tf(\w)$ denotes the gradient of the smoothed function $\tf$ at $\w$, where $\tf(\w):= \E_{\u \sim \text{Unif}(\cS_{d}(1))}[f(\w + \gamma \u)]$, and 
$C(d,\beta,c_\rho,p,r,L)$ is a constant dependent on the problem parameters $d,\beta,c_\rho,p,r,L$.
\end{restatable}

\subsection{Proof of \cref{lem:estpgd}}
\label{app:algo}

\estpgd*

\begin{proof}
We start by noting that given $\w$ and denoting $\g = o\u$:
\begin{align*}
\E_{\u,o}[\g] & =
	\E_{\u}[\E_{o}[o\u \mid \u]] 
	\\
	& = 
	\E_{\u}[ \rho\big( f(\w + \gamma\u) - f(\w - \gamma\u) \big)\u] 
\end{align*}

Let us start with the case for $p\geq 1$. Note since $\rho$ is differentiable by assumption (see \cref{assump:tf}-i) as well as $f$, using \cref{lem:fkm} in the first equality, we get:
 \begin{align}
 	\label{eq:lb9_1}
     \E_{\u} &\brk[s]!{ \rho\brk{ f(\w + \gamma \u) - f(\w - \gamma \u) } \u \dotp (\w-\w^*) } \nonumber
     \\
     &= \nonumber 
     \frac{\gamma}{d}\E_{\u}\brk[s]!{ \rho'\big( f(\w+\gamma\u) - f(\w-\gamma\u) \big) \big( \nabla f(\w+\gamma\u) + \nabla f(\w-\gamma\u) \big) } \dotp (\w-\w^*) \\
     &\geq 
     \frac{2\gamma}{d}\E_{\u}\brk[s]!{ \rhoo'\big( \abs{f(\w+\gamma\u) - f(\w-\gamma\u) }\big)\big( \nabla f(\w+\gamma\u) + \nabla f(\w-\gamma\u) \big)  } \dotp (\w-\w^*)     
     ,
 \end{align}
where the second equality follows since $\rho'(-x) = \rho'(x)$ by the anti-symmetry of $\rho$ (see \cref{assump:tf}-(i)). 
  %
 Note in \cref{eq:lb9_1}, we further lower bounded $\rho'$ by $\rhoo'$ using \cref{assump:tf}-(ii) along with the observation that $\nabla f(\w + \gamma \u) \dotp (\w-\w^*)    $ is positive for any $\u \in \cS_d(1)$ by first order optimality conditions.

It is important to note that, to ensure $r$-proximity to the origin, as needed in \cref{assump:tf}-(ii), it needed to satisfy $\abs{ f(\w+\gamma\u) - f(\w-\gamma\u) } \leq r$, which we ensured by choosing $\gamma \leq \frac{r}{2L}$ (as recall by assumption $f$ is $L$-lipschitz).

\textbf{Case: $p=1$. } Note in this case from \cref{eq:lb9_1} we get:
\begin{align*}
     \E_{\u} &\brk[s]!{ \rho\brk{ f(\w + \gamma \u) - f(\w - \gamma \u) } \u \dotp (\w-\w^*) } \nonumber
     \\
     &\geq 
     \frac{2\gamma}{d}\E_{\u}\brk[s]!{ c_\rho \big( \nabla f(\w+\gamma\u) + \nabla f(\w-\gamma\u) \big)  } \dotp (\w-\w^*)     
     \\
     & = \frac{2\gamma c_\rho}{d}\E_{\u}\brk[s]!{\big( \nabla f(\w+\gamma\u) + \nabla f(\w-\gamma\u) \big)  } \dotp (\w-\w^*) 
     \\
     & = \frac{2\gamma c_\rho}{d}\big( \nabla \E_{\u}\brk[s]!{ f(\w+\gamma\u)} + \nabla \E_{\u}\brk[s]!{ f(\w-\gamma\u )} \big) \dotp (\w-\w^*) 
     = \frac{4\gamma c_\rho}{d}{ \nabla \tf(\w)  } \dotp (\w-\w^*) 
     ,
 \end{align*}
where the inequality follows since when $p=1$, $\rhoo'(x) = c_\rho$ (independent of $x, \, \forall x \in \R$). The last equality follows by exchanging expectation and derivative (since $f$ is differentiable and finite valued by assumptions).

\textbf{Case: $p\geq 1$. } Now lets consider the case for any general $p \geq 1$: Now for the first term in \cref{eq:lb9_1}, again applying the $\beta$-smoothness of $f$ we get:
\begin{align*}
 	\gamma \u \dotp \nabla f(\w) - \tfrac12 \beta \gamma^2  &\leq  f(\w+ \gamma \u) - f(\w) \leq \gamma \u \dotp \nabla f(\w) + \tfrac12 \beta \gamma^2 \\
 	- \gamma \u \dotp \nabla f(\w) - \tfrac12 \beta \gamma^2  
 	&\leq f(\w- \gamma \u) - f(\w)\leq - \gamma \u \dotp \nabla f(\w) + \tfrac12 \beta \gamma^2 .
 \end{align*}
Subtracting the inequalities, we get
\begin{align*}
 	& \abs*{ f(\w+ \gamma \u) - f(\w-  \gamma \u) - 2 \gamma \u \dotp \nabla f(\w) } \leq  \beta \gamma^2 \\
 	\implies 
 	& 2 \gamma \u \dotp \nabla f(\w) -  \beta \gamma^2 \leq f(\w+ \gamma \u) - f(\w-  \gamma \u)  \leq 2 \gamma \u \dotp \nabla f(\w) +  \beta \gamma^2
 	.
\end{align*}
Note above implies:
$
\abs*{ f(\w+ \gamma \u) - f(\w-  \gamma \u)} \geq \abs{2 \gamma \u \dotp \nabla f(\w) } - \beta \gamma^2.
$
Now using Lem.~4 of \cite{SKM21}, we know that: 
$$
\P_\u\big( \abs{\u^\top \nabla f(\w)} \geq \beta \gamma \big) \geq 1-\lambda, \text{ where } \lambda = 
\inf_{\gamma' > 0}\bigg\{ \gamma' + \frac{2\beta\gamma \sqrt{d \log (1/\gamma')}}{\norm{\nabla f(\w)}} \bigg\}
%
$$
for any $\w \in \R^d$. Note choosing $\gamma' = \frac{\beta \gamma \sqrt{d}}{\norm{\nabla f(\w)}}$, we get $\lambda \leq \frac{\beta \gamma \sqrt{d}}{\norm{\nabla f(\x)}}\bigg( 1 + 2\sqrt{\log \frac{\norm{\nabla f(\x)}}{\sqrt d \beta \gamma}} \bigg)$. Note if we further choose $\gamma \leq \frac{\norm{\nabla f(\w)}}{10\beta \sqrt d} $, we have $\lambda \leq 1/2$.

Thus we have:
$
\P_\u\big( \abs*{ f(\w+ \gamma \u) - f(\w-  \gamma \u)}  \geq \beta \gamma^2 \big) \geq 1-\lambda,
$
and using \cref{eq:lb9_1}, we further get:
\begin{align}
 	\label{eq:lb9_2}
     \E_{\u} &\brk[s]!{ \rho\brk{ f(\w + \gamma \u) - f(\w - \gamma \u) } \u \dotp (\w-\w^*) } \nonumber
     \\
     &\geq 
     \frac{2\gamma (1-\lambda)}{d}\E_{\u}\brk[s]!{ \rhoo'\big( \beta \gamma^2\big)\big( \nabla f(\w+\gamma\u) + \nabla f(\w-\gamma\u) \big)  } \dotp (\w-\w^*)     
     ,
\end{align}
Here it is important to note that above lower bound holds since (a) from first order optimality conditions for any $\u \in \cS_d(1)$, 
$f(\w+\gamma\u)\dotp (\w-\w^*) > 0$, and also
(b) $\rhoo'\big( x \big) > 0$ for any $x > 0$ by definition of $\rhoo'$. 

Thus from \cref{eq:lb9_2} we further get:
\begin{align}
 	\label{eq:lb9_3}
     \E_{\u} &\brk[s]!{ \rho\brk{ f(\w + \gamma \u) - f(\w - \gamma \u) } \u \dotp (\w-\w^*) } \nonumber
     \\
     &\geq \nonumber
     \frac{2\gamma (1-\lambda)\rhoo'\big( \beta \gamma^2\big)}{d}\E_{\u}\brk[s]!{ \big( \nabla f(\w+\gamma\u) + \nabla f(\w-\gamma\u) \big)} \dotp (\w-\w^*)  
     \\
     & = 
      \frac{2c_\rho p(1-\lambda)\gamma^{2p-1} \beta^{p-1}}{d}\big( \nabla \tilde f(\w+\gamma\u) + \nabla \tilde f(\w-\gamma\u) \big) \dotp (\w-\w^*)       
     ,
\end{align}

The result now follows by noting $\rhoo'(x) = c_\rho p x^{p-1} \text{ for any } x \in \R_+$, $\nabla \tf(\w) = \E_{\u \sim \text{Unif}(\cS_d(1))}[\nabla f(\w + \gamma \u)]$ and choosing $\gamma= \min\Big(\frac{r}{2L},\frac{\norm{\nabla f(\w)}}{10\beta \sqrt d}\Big)$. That concludes the proof for any $p \geq 1$. 
%
\end{proof}	

\begin{rem}
It is also worth noting that, for $p=0$, $c_\rho = 1$, our transfer function recovers the $\sign$ feedback of \cite{SKM21}. In this case it can be shown that:

\begin{align*}
		\E_{\u,o}&[o \u]\dotp(\w-\w^*) 
		\geq 		
		\frac{1}{40 \sqrt d} \frac{\nabla f(\w)\dotp(\w-\w^*)}{\norm{\nabla f(\w)}},	
\end{align*}
which ssentially recovers the normalized gradient estimate (or the direction of the gradient).

The claim for $p=0$ simply follows by consecutively applying Lemma $4$ and $3$ of \cite{SKM21} as follows: From Lemma $4$ of \cite{SKM21} we have: 
\begin{align*}
		\E_{\u}&[\sign\big(f(\w+\gamma\u)-f(\w-\gamma\u)\big) \u]\dotp(\w-\w^*) 
		\geq 
		(1-\lambda) \E_{\u}[\sign\big(\nabla f(\w) \dotp \u \big) \u]\dotp(\w-\w^*),
\end{align*}
where $\lambda$ is as defined in the Case for $p \geq 1$ above. But using Lemma $3$ of \cite{SKM21} we further get:
\begin{align*}
		\E_{\u}&[\sign\big(f(\w+\gamma\u)-f(\w-\gamma\u)\big) \u]\dotp(\w-\w^*) 
		\geq 
		\frac{1}{40 \sqrt d} \frac{\nabla f(\w)\dotp(\w-\w^*)}{\norm{\nabla f(\w)}},
\end{align*}
which concludes the claim choosing $\gamma= \min\Big(\frac{r}{2L},\frac{\norm{\nabla f(\w)}}{10\beta \sqrt d}\Big)$ (noting that $\lambda \leq 1/2$ for this choice of $\gamma$ as explained in the case for $p\geq 1$ above).
\end{rem}


\input{appendix_smth2.tex}

\input{appendix_stng.tex}

\input{appendix2.tex}

}

%% file: appendix_smth2.tex
\section{Appendix for \cref{sec:analysis_smooth}}
\label{app:pgd}	

\textbf{Notations. } We denote by $\cH_t$ the history $\{\w_\tau,\u_\tau,o_\tau\}_{\tau = 1}^{t-1} \cup \w_{t}$ till time $t$, for all $t \in [T]$. $\E_t[\cdot]:= \E_{o_t,\u_t}[\cdot \mid \cH_t]$ denote the expectation with respect to $\u_t,o_t$ given $\cH_t$

\subsection{Proof of \cref{lem:pgd}}
\label{app:lempgd}

\lempgd*

\begin{proof}[\textbf{Complete Proof of \cref{lem:pgd}}]

First note by our update rule,

\begin{align}
 \label{eq:eq0}
     \E_t[\norm{\tw_{t+1}-\w^*}^2] 
     &=
     \E_t[\norm{\w_{t}-\w^*}^2] - 2{\eta} \E_t\brk[s]{ \g_t \dotp (\w_t-\w^*) } + \eta^2 \E_t\norm{\u_t}^2 \nonumber
     \\
     &=
     \E_t[\norm{\w_{t}-\w^*}^2] - 2\eta \E_t[\brk[s]{ \g_t \dotp (\w_t-\w^*) }] + \eta^2 
     .
 \end{align}

But since projection to $\cD$ reduces distance from $\w^*$ we have: 

 \[
 \norm{\w_{t+1}-\w^*}^2 \leq \norm{\tw_{t+1}-\w^*}^2
 \]

This further implies:
 \[
 \norm{\w_{t+1}-\w^*}^2 \leq \norm{\w_{t+1}'-\w_{t+1}}^2 + \norm{\w_{t+1}'-\w^*}^2 \leq \gamma^2 + \norm{\tw_{t+1}-\w^*}^2.
 \]

 Applying above in \cref{eq:eq0} we get: 
 
 \begin{align}
 \label{eq:ub1}
     \E_t[\norm{\w_{t+1}-\w^*}^2] 
     &\leq 
     \E_t[\norm{\w_{t}-\w^*}^2] - 2\eta \E_t[\brk[s]{ \g_t \dotp (\w_t-\w^*) }] + \eta^2 \nonumber
     \\
     &=
     \norm{\w_{t}-\w^*}^2 - 2\eta \E_t[\brk[s]{ \g_t \dotp (\w_t-\w^*) }] + \eta^2 
     .
 \end{align}

On the other hand, since both $f$ and $\rho$ is convex (by assumption), using \cref{lem:fkm} we get:
\begin{align}
	\label{eq:lb1}
    & \E_t [\g_t \dotp (\w_t-\w^*)] 
    = 
    \E_{\u_t} \brk[s]!{ \rho\brk{ f(\w_t + \gamma \u_t) - f(\w_t - \gamma \u_t) } \u_t \dotp (\w_t-\w^*) \mid \cH_t } \nonumber
    \\
    &= 
    \E_{\u_t} \brk[s]!{ {\rho\brk{ f(\w_t + \gamma \u_t) - f(\w_t - \gamma \u_t) }} \cdot {\u_t } \mid \cH_t } \dotp (\w_t-\w^*) \nonumber
    \\
    &= \nonumber 
    \frac{\gamma}{d}\E_{\u_t}\brk[s]!{ \rho'\big( {f(\w_t+\gamma\u_t) - f(\w_t-\gamma\u_t)} \big) \big( \nabla f(\w_t+\gamma\u_t) + \nabla f(\w_t-\gamma\u_t) \big) \mid \cH_t } \dotp (\w_t-\w^*)
    \\
    &= 
    \frac{\gamma}{d}\E_{\u_t}\brk[s]!{ \rho'\big( \abs{f(\w_t+\gamma\u_t) - f(\w_t-\gamma\u_t)} \big) \big( \nabla f(\w_t+\gamma\u_t) + \nabla f(\w_t-\gamma\u_t) \big) \mid \cH_t } \dotp (\w_t-\w^*)
    ,
\end{align}
where the last equality follows since  $\rho(-x) = -\rho(x)$ (see \cref{assump:tf}-i). 
 Now, since $f$ is convex and $\beta$-smooth, we have that:
  
 \begin{align*}
     \nabla f(\w_t+\gamma\u_t) \dotp (\w_t-\w^*)
     &\geq
     f(\w_t+\gamma\u_t) - f(\w^*+\gamma\u_t) 
     \\
     &\geq
     f(\w_t) - f(\w^*) + \nabla f(\w_t) \dotp (\gamma \u_t) - \nabla f(\w^*) \dotp (\gamma \u_t) - \beta \norm{\gamma \u_t}^2 
     \\
     &=
     f(\w_t) - f(\w^*) + \gamma (\nabla f(\w_t)  - \nabla f(\w^*)) \dotp \u_t - \beta \gamma^2
     .
 \end{align*}

Likewise, for the term $\nabla f(\w_t-\gamma\u_t) \dotp (\w_t-\w^*)$ we get,

 \begin{align*}
     \nabla f(\w_t-\gamma\u_t) \dotp (\w_t-\w^*)
     &\geq
     f(\w_t-\gamma\u_t) - f(\w^*-\gamma\u_t) 
     \\
     &\geq
     f(\w_t) - f(\w^*) - \nabla f(\w_t) \dotp (\gamma \u_t) + \nabla f(\w^*) \dotp (\gamma \u_t) - \beta \norm{\gamma \u_t}^2 
     \\
     &=
     f(\w_t) - f(\w^*) - \gamma (\nabla f(\w_t)  - \nabla f(\w^*)) \dotp \u_t - \beta \gamma^2
     .
 \end{align*}

Summing, we thus get:
 \begin{align} \label{eq:eq1}
     \big( \nabla f(\w_t+\gamma\u_t) + \nabla f(\w_t-\gamma\u_t) \big) \dotp (\w_t-\w^*)
     \geq
     2(f(\w_t) - f(\w^*)) - 2\beta \gamma^2
     .
 \end{align}

Now since by \cref{assump:tf}-(2) we have $\rho'(x)> 0$ for any $x \in (0,r]$, we can claim that 
$$\rho'\big( \abs{f(\w_t+\gamma\u_t) - f(\w_t-\gamma\u_t) }\big) > 0; $$

see \cref{rem:rhoo_justified} for a formal justification (on choices of $\gamma$ to satisfy $\abs{f(\w_t+\gamma\u_t) - f(\w_t-\gamma\u_t) } \leq r$).   
%
\begin{rem}
\label{rem:rhoo_justified}
Recall from \cref{assump:tf}, there exists some constants$p,r,c_\rho>0$ such that $\abs{\rho(x)} \geq \abs{\rhoo(x)}$ forall $x \in [-r,r]$. Clearly we apply inequality $(a)$ for $x = f(\w_t + \gamma \u_t) - f(\w_t - \gamma \u_t)$. Now to justify indeed $\abs{f(\w_t + \gamma \u_t) - f(\w_t - \gamma \u_t)} \leq r$, we note that since $f$ is $\beta$-smooth, 
it is also locally-lipschitz inside the bounded domain $\cD$ and suppose $L$ is the resulting lipschitz constant. 
Then we have $\abs{f(\w_t + \gamma \u_t) - f(\w_t - \gamma \u_t)} \leq 2L\gamma$,
 and to ensure the above condition, we can assume $2L\gamma \leq r$ by choosing small enough $\gamma$. But since we set $\gamma = \frac{\tc \epsilon}{\beta D \sqrt d}$, note the constraints are satisfied for any $\epsilon \leq \frac{r \beta D \sqrt d}{2 L \tc}$.
\end{rem}

%
For simplicity, let us denote $\E_{\u_t}[\cdot]:= \E_{\u_t}[\cdot \mid \cH_t]$. 
Recall, since $\rho'\big( \abs{f(\w_t+\gamma\u_t) - f(\w_t-\gamma\u_t) }\big) > 0$ by \cref{assump:tf}, now using \cref{eq:eq1,eq:lb1}, we can write: 

\begin{align}
\label{eq:apr}
     \nonumber \E_{t}& [\g_t \dotp (\w_t-\w^*)] 
     = 
     \frac{\gamma}{d}\E_{\u_t}\brk[s]!{ \rho'\big(\abs{ f(\w_t+\gamma\u_t) - f(\w_t-\gamma\u_t) } \big) \big( \nabla f(\w_t+\gamma\u_t) + \nabla f(\w_t-\gamma\u_t) \big) \dotp (\w_t-\w^*) }
     \\
     &
     \geq \nonumber
     \frac{2\gamma}{d} \E_{\u_t}\brk[s]!{ \rho'\big( \abs{ f(\w_t+\gamma\u_t) - f(\w_t-\gamma\u_t) } \big) \brk!{ f(\w_t) - f(\w^*) - \beta \gamma^2 } }
	     \\
     &
     \geq  \nonumber
     \frac{2\gamma}{d} \E_{\u_t}\brk[s]!{ \rho'\big( \abs{ f(\w_t+\gamma\u_t) - f(\w_t-\gamma\u_t) } \big)} \brk!{ f(\w_t) - f(\w^*) - \beta \gamma^2 } 
     \\
      & > \nonumber
     \frac{2\gamma}{d} \E_{\u_t}\brk[s]!{ \rho'\big( \abs{f(\w_t+\gamma\u_t) - f(\w_t-\gamma\u_t)}}  \big) \brk!{\epsilon - \beta \gamma^2 }  ~\big(\text{since we conditioned on } f(\w_t) - f(\w^*) > \epsilon\big),
   	\\
   	& > 
     \frac{2\gamma}{d} \E_{\u_t}\brk[s]!{ \rho'\big( \abs{f(\w_t+\gamma\u_t) - f(\w_t-\gamma\u_t)}  \big)} {\epsilon/2}  
      ,
 \end{align}
where the second last inequality follows since we choose $\gamma = \frac{\tc \epsilon}{\beta \sqrt{d}\holdon}$ (recall from \cref{thm:pgd}). However since $\tc < 1$, note we have $\gamma \leq \frac{\epsilon}{\beta \sqrt{d}\holdon}$ and further since $D> \sqrt{\frac{2\epsilon}{\beta}}$, \footnote{Note otherwise, i.e. if $D \leq \sqrt{\frac{2\epsilon}{\beta}}$, for any point $\w \in \cD$, $f(\w)-f(\w^*) \leq \epsilon$ by \cref{lem:prop_b} and the optimization problem of \cref{thm:pgd} becomes trivial to solve} above in turn implies $\gamma < \sqrt{\frac{\epsilon}{2\beta}}$ and hence $\brk!{\epsilon - \beta \gamma^2 }> \frac{\epsilon}{2}$. Last equality is from \cref{assump:tf}-(ii) where recall that $\rhoo(x) = c_\rho\sign(x)\abs{x}^p$ forall denotes the \pprox\, of $\rho$ and hence $\rho'(x)\geq c_\rho p x^{p-1} = \rhoo'(x) $ for $x \in [0,r]$.
%
Now applying the $\beta$-smoothness of $f$:

\begin{align*}
 	\gamma \u \dotp \nabla f(\w) - \tfrac12 \beta \gamma^2  &\leq  f(\w+ \gamma \u) - f(\w) \leq \gamma \u \dotp \nabla f(\w) + \tfrac12 \beta \gamma^2 \\
 	- \gamma \u \dotp \nabla f(\w) - \tfrac12 \beta \gamma^2  
 	&\leq f(\w- \gamma \u) - f(\w)\leq - \gamma \u \dotp \nabla f(\w) + \tfrac12 \beta \gamma^2 .
 \end{align*}

Subtracting the inequalities, we get

\begin{align*}
 	& \abs*{ f(\w+ \gamma \u) - f(\w-  \gamma \u) - 2 \gamma \u \dotp \nabla f(\w) } \leq  \beta \gamma^2 \\
 	\implies 
 	& 2 \gamma \u \dotp \nabla f(\w) -  \beta \gamma^2 \leq f(\w+ \gamma \u) - f(\w-  \gamma \u)  \leq 2 \gamma \u \dotp \nabla f(\w) +  \beta \gamma^2
 	.
 \end{align*}

Note above implies:
 $
 \abs*{ f(\w_t+ \gamma \u_t) - f(\w_t-  \gamma \u_t)} \geq \abs{2 \gamma \u_t \dotp \nabla f(\w_t) } - \beta \gamma^2.
 $
 
 Now taking expectation over $\u_t$ in both side:
 \begin{align*}
 \E_{\u_t}[\abs*{ f(\w_t+ \gamma \u_t) - f(\w_t-  \gamma \u_t)}] \geq \E_{\u_t}[\abs{2 \gamma \u_t \dotp \nabla f(\w_t) }] - \beta \gamma^2
 = 2\frac{\tc \gamma \norm{\nabla f(\w_t)} }{\sqrt{d}} - \beta \gamma^2
 \end{align*}
 where the last equality is due to \cref{lem:normgrad}. 
 Additionally, since we assumed $f(\w_t)-f(\w^*) > \epsilon$, i.e. the suboptimality gap to be at least $\epsilon$, by \cref{lem:gradf_lb} we can further derive a lower bound: 
 \begin{align*}
 \E_{\u_t}[\abs*{ f(\w_t+ \gamma \u_t) - f(\w_t-  \gamma \u_t)}] \geq \frac{2\tc \gamma \epsilon}{\sqrt{d}\norm{\w_t - \w^*}} - \beta \gamma^2 
 = {\frac{2\tc \gamma \epsilon}{\sqrt{d}D}} - \beta \gamma^2. 
 \end{align*}
 
 And now note that setting $\gamma = \frac{\tc \epsilon}{\beta \sqrt{d}\holdon}$, the right hand side is positive. Now, lower bounding $\rho'$ by $\rhoo'$ as per \cref{assump:tf}-(ii) and further applying monotonicity of $\rhoo'(\cdot)$ in the positive orthant, we get:
 \begin{align}
 \label{eq:may15a}
 \rho'(\E_{\u_t}[\abs*{ f(\w_t+ \gamma \u_t) - f(\w_t -  \gamma \u_t)}]) \geq \rho'(\E_{\u_t}[\abs*{ f(\w_t+ \gamma \u_t) - f(\w_t -  \gamma \u_t)}]) \geq \rhoo'\bigg( {\frac{2\tc \gamma \epsilon}{\sqrt{d}D}} - \beta \gamma^2 \bigg).
 \end{align}

Then from \cref{eq:apr}:
 \begin{align}
 	\label{eq:lb_lin}
 	\E_{t} & [\g_t \dotp (\w_t-\w^*)] \nonumber
 	\geq \nonumber
 	\frac{2\gamma}{d} \rhoo'\bigg( {\frac{2\tc \gamma \epsilon}{\sqrt{d}D}} - \beta \gamma^2 \bigg) \nicefrac{\epsilon}{2}
 	~~(\text{applying the bound from from } \eqref{eq:may15a})		
 	\\
 	&\geq  
 	\frac{2\gamma}{d}  \Big( c_\rho p \gamma^{p-1} \abs{\frac{2\tc \epsilon}{\sqrt{d}\holdon} -  \beta \gamma}^{p-1} \Big) \nicefrac{\epsilon}{2}
 	~~(\text{as, } \rhoo'(x) = c_\rho p x^{p-1} \text{ for any } x \in \R_+) 
 	,
 \end{align}
 Then for the above choice of $\gamma = \frac{\tc \epsilon}{\beta \sqrt{d}\holdon}$, we finally get: 

\begin{align}
 	\label{eq:lb4}
 	\E_{t} [\g_t \dotp (\w_t-\w^*)] 
 	\nonumber &\geq 
 	\frac{2\gamma \epsilon}{2d}  \big( c_\rho p \gamma^{p-1} \abs{\frac{ \tc^2\epsilon^2}{{d}\holdon^2}}^{p-1} \big) 
 	\\
 	&= 
 	\frac{pc_\rho \tc^{2p-1} \epsilon^{2p}}{d^{(2p+1)/2}\beta^p \holdon^{2p-1}}
 	,
 \end{align}

Combining \cref{eq:ub1} with \cref{eq:lb4}:

\begin{align}
    \E_t & [\norm{\w_{t+1}-\w^*}^2] 
    \leq 
    \norm{\w_{t}-\w^*}^2 - 2\eta \E_t[\brk[s]{ \g_t \dotp (\w_t-\w^*) }] + \eta^2 \nonumber
    \\
    & \leq \nonumber 
    \norm{\w_{t}-\w^*}^2
     - 2\eta \Bigg( \frac{pc_\rho \tc^{2p-1} \epsilon^{2p}}{d^{(2p+1)/2}\beta^p \holdon^{2p-1}} \Bigg) + \eta^2
     \\
    & \leq  
    \norm{\w_{t}-\w^*}^2
    - \Bigg( \frac{pc_\rho \tc^{2p-1} \epsilon^{2p}}{d^{(2p+1)/2}\beta^p \holdon^{2p-1}} \Bigg)^2, ~~\text{ setting } \eta = \Bigg( \frac{pc_\rho \tc^{2p-1} \epsilon^{2p}}{d^{(2p+1)/2}\beta^p \holdon^{2p-1}} \Bigg) \nonumber 
    ,
\end{align}
which concludes the claim of \cref{lem:pgd}.
\end{proof}

\subsection{Proof of \cref{thm:pgd}}

\thmpgd*

\begin{proof}[\textbf{Complete Proof of \cref{thm:pgd}}]
 We denote by $\cH_t$ the history $\{\w_\tau,\u_\tau,o_\tau\}_{\tau = 1}^{t-1} \cup \w_{t}$ till time $t$.
 
 We start by noting that by definition: 
 \[
 \E_{o_t}[\g_t \mid \cH_t,\u_t] = \rho(f(\w_t+\gamma \u_t) - f(\w_t - \gamma \u_t))\u_t
 \]

We proceed with the proof inductively, i.e. given $\cH_t$ and assuming (conditioning on) $f(\w_t) - f(\w^*) > \epsilon$, we can show that $\w_{t+1}$ always come closer to the minimum $\w^*$ on expectation in terms of the $\ell_2$-norm. More formally, given $\cH_t$ and assuming  $f(\w_t) - f(\w^*) > \epsilon$ we will show:

 \[
 \E_t[\norm{\w_{t+1}-\w^*}^2] \leq \norm{\w_{t}-\w^*}^2
 \]

where $\E_t[\cdot]:= \E_{o_t,\u_t}[\cdot \mid \cH_t]$ denote the expectation with respect to $\u_t,o_t$ given $\cH_t$.  
 The precise statement is given by \cref{lem:pgd}. 

Given the statement of \cref{lem:pgd}, now note that by iteratively taking expectation over $\cH_T$ on both sides of \cref{eq:fin1orig} and summing over $t=1,\ldots,T$, we get,

\begin{align*}
	\E_{\cH_T} [\norm{\w_{T+1}-\w^*}^2] 
	& \leq 
	\norm{\w_{1}-\w^*}^2
	- \frac{p^2 (\tc^{2p-1} c_\rho)^2\epsilon^{4p}}{d^{2p+1}\beta^{2p} \holdon^{4p-2}}T	 
	.
\end{align*}
However, note if we set $T = \frac{d^{2p+1}\beta^{2p} \holdon^{4p}}{p^2 (\tc^{2p-1} c_\rho)^2\epsilon^{4p}}$,
this implies 
$\E_{\cH_T} [\norm{\w_{T+1}-\w^*}^2] \leq 0$, or equivalently $\w_{T+1} = \w^*$, which concludes the claim. 

To clarify further, note we show that for any run of Alg. \ref{alg:pgd} if indeed $f(\w_{t}) - f(\w^*) > \epsilon$ continues to hold for all $t = 1,2, \ldots T$, then $\w_{T+1} = \w^*$ at $T = \frac{d^{2p+1}\beta^{2p} \holdon^{4p}}{p^2 (\tc^{2p-1} c_\rho)^2\epsilon^{4p}}$. If not, there must have been a time $t \in [T]$ such that $f(\w_t) - f(\w^*) < \epsilon$. This concludes the proof with $T_\epsilon = T$.
\end{proof}


\subsection{Proof of \cref{thm:pgdi}}
\label{app:pgdi}

\thmpgdi*

\begin{proof}
\textbf{Case 1. ~Linear transfer functions: $\rho(x) = c_\rho x, ~\forall x \in \R_+$}. So in this case $\rho$ is the \pprox\, of itself, i.e. $\rhoo = \rho$ with $p=1$ and any $r \in \R$.
  
 Here $\rho'(x) = c_\rho$ for any $x \in \R$. 
 Then following the same steps as derived in the proof of \cref{thm:pgd}, note we have,

\begin{align*}
     \E_{t}[\g_t \dotp (\w_t-\w^*)] 
     & \geq  
     \frac{2\gamma}{d} \E_{\u_t}\brk[s]!{ \rho'\big( {f(\w_t+\gamma\u_t) - f(\w_t-\gamma\u_t)}}  \big) \brk!{ f(\w_t) - f(\w^*) - \beta \gamma^2 } 
     \\
     & = \frac{2\gamma}{d} c_\rho \brk!{ f(\w_t) - f(\w^*) - \beta \gamma^2 } 
 \end{align*}
 
Moreover, since we conditioned on $f(\w_t) - f(\w^*) > \epsilon$, plugging this in above and combining with 

\begin{align*}
     \E_t [\norm{\w_{t+1}-\w^*}^2] 
     & \leq 
     \norm{\w_{t}-\w^*}^2
      - \frac{4 \eta \gamma c_\rho}{d}\Bigg(
 	 \brk!{ f(\w_t) - f(\w^*) }  - \beta \gamma^2 \Bigg) + \eta^2
 	 \\
 	 & \leq 
     \norm{\w_{t}-\w^*}^2
      - \frac{4 \eta \gamma c_\rho}{d}\Bigg(
 	 \epsilon  - \beta \gamma^2 \Bigg) + \eta^2
 	 \\
 	 & \overset{(a)}{=} 
     \norm{\w_{t}-\w^*}^2
      - \frac{4 \eta c_\rho}{d}\big( \frac{\epsilon^{3/2}}{2\sqrt{2 \beta}} \big) + \eta^2	 
 	 \\
 	 & \overset{(b)}{=} 
     \norm{\w_{t}-\w^*}^2
      - \frac{c_\rho^2 \epsilon^3}{2\beta d^2}	 
 	 , 
 \end{align*}
 where $(a)$ follows by setting $\gamma = \sqrt{\frac{\epsilon}{2\beta }}$, and $(b)$ follows by setting $\eta= \frac{c_\rho \epsilon^{3/2}}{d \sqrt{2\beta}}$.

Then same as the proof of \cref{thm:pgd}, now iteratively taking expectations over $\cH_T$ on both sides of \cref{eq:fin1orig} and summing over $t=1,\ldots,T$, we get,

\begin{align*}
	\E_{\cH_T} [\norm{\w_{T+1}-\w^*}^2] 
	& \leq 
	\norm{\w_{1}-\w^*}^2
	- \frac{c_\rho^2 \epsilon^3}{2\beta d^2}	T	 
	.
\end{align*}
However, note if we set $T = \frac{2 d^{2}\beta \holdon^{2}}{c_\rho^2\epsilon^{3}}$,
this implies 
$\E_{\cH_T} [\norm{\w_{T+1}-\w^*}^2] \leq 0$, or equivalently $\w_{T+1} = \w^*$, which concludes the claim (similarly as line of argument we concluded the proof of \cref{thm:pgd}). 

\noindent 
\textbf{Case 2. ~Sigmoid transfer functions: $\rho(x) = \frac{1-e^{-\omega x}}{1+e^{-\omega x}}, ~\forall x \in \R_+$, $\omega > 0$}.  
In this case the it can be shown that $\rho$ can be approximated by a linear function near the origin, or more specifically, depending on the constant $\omega$, there exists $c_\rho^\omega$ and $r^\omega$ such that
\[
\abs{\rho(x)} \geq c_\rho^\omega \abs{x}  ~\forall x \in [-r^\omega,r^\omega].
\] 
The claimed convergence bound now follows similar to the analysis shown for Case $1$ above.

 \noindent 
 \textbf{Case 3. ~Sign transfer function: $\rho(x) = \sign(x), ~\forall x \in \R_+$}.
 
In this case also, $\rho$ is the \pprox\, of itself,  with $p=0$, $c_\rho = 1$ and any $r \in \R$. This particular transfer function was considered in the similar optimization setup in \cite{SKM21}. We show below how our proposed \cref{alg:pgd} (\kgd) generalizes their $\beta$-NGD algorithm and recovers their convergence rate of $O(\epsilon^{-1})$. 

We start by noting that, our algorithm generalizes the $\beta$-NGD algorithm of \cite{SKM21}. 
The convergence rate claim now follows by noting that in this case our descent direction $\g_t$ at any point $\w_t$, becomes the normalized gradient estimate, of $\w_t$ with high probability (over the random draws of $\u_t$). Roughly speaking it can be show that
\[
\E_{\u_t}[\g_t \mid \w_t] \approx \frac{\tc }{\sqrt d}\big( \nicefrac{\nabla f(\w_t)}{\norm{\nabla f(\w_t)}} \big), 
\]
using Lemma $2$ and $4$ of \cite{SKM21}, or more precisely,
\[
\E_{\u_t}[\g_t \dotp (\w_t-\w^*) \mid \w_t]
	\geq  \frac{\tc }{\sqrt d}\bigg( \frac{\nabla f(\w_t)\dotp (\w_t-\w^*)}{\norm{\nabla f(\w_t)}} \bigg) - 2\lambda  \norm{\w_t-\w^*}
	,
\]
where $\lambda=\frac{3\beta\delta}{\|\nabla f(\w_t)\|} \sqrt{d\log\frac{\|\nabla f(\w_t)\|}{\sqrt{d}\beta\delta}}$.  
 Combining the above bound in the proof of \cref{thm:pgd} (to lower bound the term $\E_t[\g_t\cdot (\w_t - \w^*)]$), the result follows. In fact in this case the proof of \cref{thm:pgd} exactly follows the same line of argument as that of the proof of Theorem $5$ of \cite{SKM21}. This shows the generalization ability of our proof analysis for different special class of transfer functions. 
\end{proof}

%% file: appendix_stng.tex
\section{Appendix for \cref{sec:analysis_strong}}
\label{app:epgd}	

\subsection{Proof of \cref{lem:epgd}}
\label{app:lemepgd}	

\lemepgd*

 \begin{proof}[\textbf{Complete Proof of \cref{lem:epgd}}]
 The proof relies on analyzing the epochwise performance  guarantee of any representative run \kgd$\big(\w_{k},\eta_k, \gamma_k,t_k\big)$ (see Line \#5 of \cref{alg:epgd}).

For simplicity of notations, for any epoch $k$, inside the call of \kgd$\big(\w_{k},\eta_k, \gamma_k,t_k\big)$, let us assume $\x_1 (= \w_k)$ denotes the initial point in the run of \kgd~(\cref{alg:pgd}) and let is denote by $D_0 = \norm{\x_1-\w^*}$. The goal is to analyze the guarantees on the output point $\x_{T+1}$ of the run of \pgd\, after $T = t_k$ time steps; thus $\x_{T+1} = \w_{k+1}$. We use the same notations as used in the proof of \cref{thm:pgd}. 

Recall from \cref{eq:ub1}, at any time step $t$ inside the run of \pgd\, we have: 
 \begin{align}
 \label{eq:first_ab1}
     \E_t[\norm{\x_{t+1}-\w^*}^2] 
     &\leq 
     \norm{\x_{t}-\w^*}^2 - 2\eta \E_t\brk[s]{ \g_t \dotp (\x_t-\w^*) } + \eta^2 
     ,
 \end{align}
 and on the other hand, from \cref{eq:lb1} we have:
 \begin{align}
 \label{eq:first_ab2}
     & \E_t [\g_t \dotp (\x_{t+1}-\w^*)] 
     = 
     \E_{\u_t} \brk[s]!{ \rho\brk{ f(\x_t + \gamma \u_t) - f(\x_t - \gamma \u_t) } \u_t \mid \cH_t } \dotp (\x_t-\w^*) \nonumber
     \\
     &=
     \frac{\gamma}{d}\E_{\u_t}\brk[s]!{ \rho'\big( \abs{f(\x_t+\gamma\u_t) - f(\x_t-\gamma\u_t)} \big) \big( \nabla f(\x_t+\gamma\u_t) + \nabla f(\x_t-\gamma\u_t) \big) \mid \cH_t } \dotp (\x_t-\w^*) 
     .
 \end{align}
 Now since $f$ is convex and $\beta$-smooth, applying \cref{lem:prop_ab}, we know that for any $\w \in \cD$,  
 \begin{align*}
 \big(\nabla f(\w) - \nabla f(\w^*)\big)^\top(\w-\w^*) \geq   \frac{\alpha\beta}{\alpha+\beta}\norm{\w-\w^*}^2
 + 
 \frac{1}{\alpha+\beta}\norm{\nabla f(\w) - \nabla f(\w^*)}^2.
 \end{align*}	
 Moreover since $\w^*$ is the minimizer of $f$ in $\cD$, we have $\nabla f(\w^*) \dotp (\x_t-\w^*) \geq 0$ (from first order optimality conditions, see \cref{lem:foo}). 
 Hence from above we further get: 

 \begin{align*}
 	\nabla f(\x_t+\gamma\u_t) & \dotp (\x_t + \gamma\u_t-\w^*)
 	\geq 
 	\big(\nabla f(\x_t+\gamma\u_t) - \nabla f(\w^*) \big) \dotp (\x_t + \gamma\u_t-\w^*)\\
 	& \geq   
 \frac{\alpha\beta}{\alpha+\beta}\norm{\x_t+\gamma\u_t-\w^*}^2
 + 
 \frac{1}{\alpha+\beta}\norm{\nabla f(\x_t+\gamma\u_t) - \nabla f(\w^*)}^2.
 \end{align*}

 Likewise,
 \begin{align*}
 	\nabla f(\x_t-\gamma\u_t) & \dotp (\x_t-\gamma\u_t-\w^*)
 	\geq   
 \frac{\alpha\beta}{\alpha+\beta}\norm{\x_t-\gamma\u_t-\w^*}^2
 + 
 \frac{1}{\alpha+\beta}\norm{\nabla f(\x_t-\gamma\u_t) - \nabla f(\w^*)}^2.
 \end{align*}

Combining the two, we get:
 \begin{align*} 
     & \nabla f(\x_t+\gamma\u_t)\dotp (\x_t+\gamma\u_t-\w^*) + \nabla f(\x_t-\gamma\u_t) \dotp (\x_t-\gamma\u_t-\w^*)
     \\
     & \geq
 \frac{\alpha\beta}{\alpha+\beta}(\norm{\x_t+\gamma\u_t-\w^*}^2+\norm{\x_t-\gamma\u_t-\w^*}^2)
 	\\
 	& \quad\quad\quad\quad\quad\quad + \frac{1}{\alpha+\beta}(\norm{\nabla f(\x_t+\gamma\u_t) - \nabla f(\w^*)}^2+\norm{\nabla f(\x_t-\gamma\u_t) - \nabla f(\w^*)}^2)
 	\\
 	& \overset{(a)}{\geq}
 \frac{\alpha\beta}{\alpha+\beta}(\norm{\x_t+\gamma\u_t-\w^*}^2+\norm{\x_t-\gamma\u_t-\w^*}^2) 
 	+ \frac{\alpha}{\alpha+\beta}(\norm{\x_t+\gamma\u_t-\w^*}^2+\norm{\x_t-\gamma\u_t-\w^*}^2)  
 	\\
 	& =
 \frac{2(\alpha\beta + \alpha)}{\alpha+\beta}(\norm{\x_t-\w^*}^2+\gamma^2\norm{\u_t}^2) = \frac{2(\alpha\beta + \alpha)}{\alpha+\beta}(\norm{\x_t-\w^*}^2+\gamma^2)	
 	,
 \end{align*}
 where the inequality $(a)$ follows due to \cref{lem:prop_a}. Thus we can write: 

\begin{align} 
 \label{eq:eq12}
     \nabla f(\x_t+ & \gamma\u_t)\dotp (\x_t-\w^*) + \nabla f(\x_t-\gamma\u_t) \dotp (\x_t-\w^*) \nonumber
     \\
     & \geq  \nonumber 
  \frac{2(\alpha\beta + \alpha)}{\alpha+\beta}(\norm{\x_t-\w^*}^2+\gamma^2)	- \gamma (\nabla f(\x_t+\gamma\u_t)-\nabla f(\x_t-\gamma\u_t))\dotp \u_t
     \\
     & \overset{(b)}{\geq} \nonumber
     \frac{2(\alpha\beta + \alpha)}{\alpha+\beta}(\norm{\x_t-\w^*}^2+\gamma^2)- \gamma \norm{\nabla f(\x_t+\gamma\u_t)-\nabla f(\x_t-\gamma\u_t)}
 	\\
 	& \geq 
 	\frac{2(\alpha\beta + \alpha)}{\alpha+\beta}(\norm{\x_t-\w^*}^2+\gamma^2)- 2\beta \gamma^2 \nonumber
 	\\
 	& = \frac{2\alpha(\beta + 1)}{\beta(\tkappa+1)}\norm{\x_t-\w^*}^2
 	- \frac{2(\beta - \tkappa)}{\tkappa+1}\gamma^2	
 	,
 \end{align}
 where $(b)$ follows from Cauchy-Schwarz, the last inequality holds due to \cref{lem:prop_b}, and $\tkappa:=\frac{\alpha}{\beta}$.

Recall we assumed $\norm{\x_1-\w^*} = D_0$. Let $T (=t_k)$ be the sample complexity of the run \kgd$\big(\w_k,\eta_k, \gamma_k,t_k\big)$, i.e. the \kgd\, runs for $T$ time-steps starting from the initial point $\x_1 = \w_k$. 

\textbf{Case analysis 1: (Assume $\norm{\x_t - \w^*}\geq D_0/2$ for all $t = [T]$). } In this case, by assumption, $\norm{\x_t-\w^*} \geq D_0/2$. So from above we further get: 

\begin{align*}
 \E_{\u_t}[\abs*{ f(\x_t+ \gamma \u_t) - f(\x_t-  \gamma \u_t)}] \geq \frac{2\tc \alpha \gamma D_0 }{2\sqrt{d}} - \beta \gamma^2. 
 \end{align*}
 And now note that setting $\gamma = \frac{\tc \alpha D_0}{2\beta \sqrt{d}}$, the right hand side is positive. 
 Further using \cref{assump:tf}-(3) and by the definition of $\rhoo$, we get:
 \begin{align}
 \label{eq:may15}
\E_{\u_t}[\rho'(\abs*{ f(\x_t+ \gamma \u_t) - f(\x_t -  \gamma \u_t)})] & \geq  \nonumber 
 \E_{\u_t}[\rhoo'(\abs*{ f(\x_t+ \gamma \u_t) - f(\x_t -  \gamma \u_t)})] \\
 & \geq  \nonumber
 \rhoo'(\E_{\u_t}[\abs*{ f(\x_t+ \gamma \u_t) - f(\x_t -  \gamma \u_t)}]) \\
 & \geq \rhoo'\bigg( \frac{2\tc \alpha \gamma D_0 }{2\sqrt{d}} - \beta \gamma^2 \bigg) = \rhoo'\bigg( \frac{\tc^2 \alpha^2 D_0^2 }{4\beta d} \bigg).
 \end{align}
where the second and the third inequalities are respectively due to Jensen's inequality, as by definition $\rhoo'$ is convex (for any $p \geq 1$) in the positive orthant and also monotonically increasing.  
 Further by denoting $A = \frac{\alpha(\beta + 1)}{\beta(\tkappa+1)}$ and $A'= \frac{(\beta - \tkappa)}{\tkappa+1}$, note that for this choice of $\gamma$ we have 
 \begin{align*}
 A\norm{\x_t-\w^*}^2	- A'\gamma^2	 
 	& > 
 	\frac{\alpha(\beta + 1)}{\beta(\tkappa+1)}\frac{D_0^2}{4} - \frac{(\beta - \tkappa)}{\tkappa+1} \bigg(\frac{\tc \alpha D_0}{2\beta \sqrt{d}} \bigg)^2
 	\\
 	& \overset{(c)}{\geq} 
 	\frac{\alpha(\beta + 1)}{\beta(\tkappa+1)}\frac{D_0^2}{4} - \frac{(\beta - \tkappa)}{\tkappa+1} \bigg(\frac{\alpha^2 D_0^2}{4\beta^2 {d}} \bigg)
 	\\
 	& \overset{}{\geq} 
 	\frac{\alpha(\beta + 1)}{\beta(\tkappa+1)}\frac{D_0^2}{4} - \frac{\alpha^2}{\beta {d} (\tkappa+1)} \frac{ D_0^2}{ 4} 
 	\\
 	& \geq
 	\frac{\alpha(\beta + 1)}{\beta(\tkappa+1)}\frac{D_0^2}{4} - \frac{\alpha^2}{\beta {d} (\tkappa+1)} \frac{ D_0^2}{ 4}
 	\\
 	& = 
 	\frac{\alpha(\beta + 1 - \alpha/d)}{\beta(\tkappa+1)}\frac{D_0^2}{4} > \frac{\alpha}{\alpha+\beta}\frac{D_0^2}{4} 
 	,
 \end{align*}
 where $(c)$ follows since $\tc \leq 1$ (as follows form \cref{lem:normgrad}), and the last inequality is due to the fact that by definition $\beta \geq \alpha$ and $d \geq 1$. Then combining above with \cref{eq:first_ab2,{eq:may15}} we have:

\begin{align}
 	\label{eq:lb_lin}
 	\E_{t} & [\g_t \dotp (\x_t-\w^*)] \nonumber
 	\geq \nonumber
 	\frac{2\gamma}{d} \rhoo'\big( \E_{\u_t}\brk[s]!{\abs{f(\x_t+\gamma\u_t) - f(\x_t-\gamma\u_t)}} \big) \bigg(  \frac{\alpha}{\alpha+\beta}\frac{D_0^2}{4} \bigg) 
 	\\
 	& \geq \nonumber
 	\frac{2\alpha D_0^2}{d (\alpha+\beta)4} \bigg(\frac{\tc \alpha D_0}{2\beta \sqrt{d}} \bigg) \rho'\bigg( \frac{\tc^2 \alpha^2 D_0^2 }{4\beta d}  \bigg)
 	~~(\text{from } \eqref{eq:may15}, \text{ where recall we needed to set } \gamma = \frac{\tc \alpha D_0}{2\beta \sqrt{d}})	
 	\\
 	& = \nonumber
 	\frac{2\alpha D_0^2}{d (\alpha+\beta)4} \bigg(\frac{\tc \alpha D_0}{2\beta \sqrt{d}} \bigg) c_\rho p\bigg( \frac{\tc^2 \alpha^2 D_0^2 }{4\beta d}  \bigg)^{p-1}
 	~~(\text{as, } \rho'(x) = c_\rho p x^{p-1} \text{ for any } x \in \R_+) 
 	\\
 	& =
 	B D_0^{2p+1}, ~~\bigg(\text{where we denote by, } B:= \frac{2c_\rho p}{(\alpha+\beta)}\Big( \big(\nicefrac{\alpha^2}{4\beta} \big)^{p} \frac{\tc^{2p-1} }{d^{\frac{2p+1}{2}}} \Big) \bigg)
 	.
 \end{align}

Plugging the above expression in \cref{eq:first_ab1}:

\begin{align}
 \label{eq:fin1ab}
     \E_t & [\norm{\x_{t+1}-\w^*}^2] \nonumber
     \leq 
     \norm{\x_t-\w^*}^2 - 2\eta \E_t[\brk[s]{ \g_t \dotp (\x_t-\w^*) }] + \eta^2 
     \\
     & \leq \nonumber 
     D_0^2
      - \eta B (D_0^{2p+1}) + \eta^2 ~~\bigg(\text{ where recall } B:= \frac{2c_\rho p}{(\alpha+\beta)}\Big( \big(\nicefrac{\alpha^2}{4\beta} \big)^{p} \frac{\tc^{2p-1} }{d^{\frac{2p+1}{2}}} \Big)   \bigg) 	\\
     & = 
     D_0^2 - B^2 {D_0^2}^{(2p+1)}, ~~\bigg(\text{ by setting } \eta = BD_0^{2p+1} = \frac{2c_\rho p}{(\alpha+\beta)}\Big( \big(\nicefrac{\alpha^2}{4\beta} \big)^{p} \frac{\tc^{2p-1} }{d^{\frac{2p+1}{2}}} \Big) D_0^{2p+1} \bigg)
 	.  
 \end{align}

Now let us fix some time-stamp $T$ and let us assume $\norm{\x_t - \w^*}\geq D_0/2$ for all $t = [T]$. 
 Then taking expectation over $\cH_{T+1}$ on both sides iteratively and summing over $t=1,\ldots,T$, note that we get:

\begin{align*}
 	\E_{\cH_{T+1}} [\norm{\x_{T+1}-\w^*}^2] 
 	& \leq 
 	D_0^2	- B^2 {D_0^2}^{(2p+1)}T	 
 	.
 \end{align*}

 Then if we set $T = \frac{1}{2B^2{D_0^{2}}^{2p}}$, at time $T+1$ we have:
 \[
 \E_{\cH_{T+1}} [\norm{\x_{T+1}-\w^*}^2] \leq D_0^2/2 < 3D_0^2/4.
 \]

\textbf{Case analysis 2: ($\exists$ at least an $\tau \in [T]$ such that $\norm{\x_{\tau}-\w^*} \leq D_0/2$).}
 In this case, from \cref{eq:first_ab1}, after $(T-\tau)$ steps we can have the expected value of $\norm{\x_{T+1}-\w^*}^2$ can be at most: 
 \[
 \E_{\cH_{T+1}} [\norm{\x_{T+1}-\w^*}^2] \leq \norm{\x_{\tau}-\w^*}^2 + (T-\tau)\eta^2 \leq D_0^2/4 + T\eta^2  = 3D_0^2/4,
 \]
 since recall that we set $\eta = BD_0^{2p+1} = \frac{2c_\rho p}{(\alpha+\beta)}\Big( \big(\nicefrac{\alpha^2}{4\beta} \big)^{p} \frac{\tc^{2p-1} }{d^{\frac{2p+1}{2}}} \Big) D_0^{2p+1}$.

This implies that for the above choice of $\gamma$ and $\eta$ we can get constant fraction reduction in the ``sub-optimality gap" $(\norm{\w_k-\w^*})$ after at most $O(\frac{1}{2B^2{D_0^{2}}^{2p}})$ time steps.
 \end{proof}

\subsection{Proof of \cref{thm:epgd}}

\thmepgd*

\begin{proof}[\textbf{Complete Proof of \cref{thm:epgd}}]
The proof of  \cref{thm:epgd} is based on the key claim of \cref{lem:epgd} which shows that after every epoch of length $t_k$, the distance of the resulting point $\w_{k+1}$ from the optimal $\w^*$ must decrease by at least a constant fraction. Recall from the statement of \cref{lem:epgd}, we have:

\lemepgd*

The claim of \cref{thm:epgd} now follows form the following \textbf{epoch-wise recursion argument: }

Let $\cH_{[k]}:= \{(\w_{k'})_{k'\in [k_\epsilon]},(\u_{t'},o_{t'})_{t' \in [\sum_{k'= 1}^{k}t_{k'}]}\} \cup \{\w_{k+1}\}$ denotes the complete history till the end of epoch $k$ starting from the first epoch $\forall k \in [k_\epsilon]$. 

Further, let us denote by $\cH_{k}:= \{\w_{k},(\u_{t'},o_{t'})_{t' \in [\sum_{k'= 1}^{k-1}t_{k'}+1,\sum_{k'= 1}^{k}t_{k'}]}\} \cup \{\w_{k+1}\}$ be the history only within epoch $k$. 

\textbf{Proof of Correctness. } From \cref{lem:epgd}, note we have already established that 
$E_{\cH_k}[\norm{\w_{k+1}-\w^*}^2 \mid \w_k] \leq {\frac{3}{4}}\norm{\w_{k}-\w^*}^2$.

Applying the argument iteratively over $E$ epochs, and the law of iterated expectations, we have:
\begin{align}
\label{eq:may19a}
\E_{\cH_{[E]}}[\norm{\w_{E+1}-\w^*}^2 ] \leq (\nicefrac{3}{4})^E\norm{\w_{1}-\w^*}^2.
\end{align}

Thus choosing $E = \lceil \log_{4/3} (\frac{\beta D^2}{2\epsilon}) \rceil$, where $\norm{\w_1 - \w^*} \leq D$,
we have 
\begin{align*}
E \geq \log_{4/3} (\frac{\beta \norm{\w_1 - \w^*}^2}{2\epsilon})
	~
	\implies 
	~
	(\nicefrac{3}{4})^{E}\norm{\w_1 - \w^*}^2 \leq \frac{2\epsilon}{\beta}.
\end{align*}

Thus from \eqref{eq:may19a}, we get: 
\[
\E_{\cH_{[E]}}[\norm{\w_{E+1}-\w^*}^2 ] \leq \frac{2\epsilon}{\beta},
\]
and further applying $\beta$-smoothness of $f$, we get:
\begin{align*}
\E_{\cH_{[E]}}[f(\w_{E+1})-f(\w^*) ] \leq \E_{\cH_{[E]}}[\frac{\beta}{2}\norm{\w_{E+1}-\w^*}^2 ] \leq \epsilon,
\end{align*}
which proves the correctness of \cref{alg:epgd} for the choice of total number of epochs $k_{\epsilon}=E$.

\textbf{Proof of Sample Complexity. } In order to verify that \cref{alg:epgd} indeed converges to an $\epsilon$-optimal point in $O(\epsilon^{-2p})$ sample complexity, note we simply need to count the total sample complexity incurred in the $k_\epsilon$ epochwise runs of \kgd\, (see Line \#5 of \cref{alg:epgd}). However, by design of \ekgd\,(\cref{alg:epgd}), since \kgd$\big(\w_{k},\eta_k, \gamma_k,t_k\big)$ is run for only $t_k = \frac{1}{2B^2{(D_k^{2})}^{2p}}$ iterations, the total sample complexity of \ekgd\, becomes:

\begin{align*}
\sum_{k = 1}^{k_\epsilon} t_k 
	&= \frac{1}{2B^2}\sum_{k = 1}^{k_\epsilon} \frac{1}{{(D_k^{2})}^{2p}} 
	= \frac{1}{2B^2}\bigg( \frac{1}{{(D_1^{2})}^{2p}} + \frac{1}{{(D_2^{2})}^{2p}} + \cdots + \frac{1}{{(D_{k_\epsilon}^{2})}^{2p}} \bigg)
	\\
	& = \frac{1}{2B^2}\bigg( \frac{1}{{(D^{2})}^{2p}} + \frac{1}{{(\nicefrac{3}{4}D^{2})}^{2p}} + \frac{1}{{((\nicefrac{3}{4})^2 D^{2})}^{2p}} + \cdots + \frac{1}{{((\nicefrac{3}{4})^{k_\epsilon-1}D_{k_\epsilon}^{2})}^{2p}} \bigg)
	\\
	& = \frac{1}{2B^2 (D^2)^{2p}}\bigg( 1 + \frac{1}{(\nicefrac{3}{4})^{2p}} + \frac{1}{{((\nicefrac{3}{4})^2}^{2p}} + \cdots + \frac{1}{{((\nicefrac{3}{4})^{k_\epsilon-1})}^{2p}} \bigg)
	\\
	& = \frac{1}{2B^2 (D^2)^{2p}}\bigg( 1 + \frac{1}{(\nicefrac{3}{4})^{2p}} + \frac{1}{((\nicefrac{3}{4})^{2p})^2} + \cdots + \frac{1}{((\nicefrac{3}{4})^{2p})^{k_\epsilon-1}} \bigg)
	\\
	& = \frac{1}{2B^2 (D^2)^{2p}}\frac{(\nicefrac{4}{3}^{2p})^{k_\epsilon} -1}{\nicefrac{4}{3}^{2p}-1} \leq  \frac{1}{4B^2 (D^2)^{2p}}\Big( (\nicefrac{\beta D^2}{2 \epsilon})^{2p} -1 \Big) = O\Big(\frac{1}{B^2\epsilon^{2p}}\Big)
\end{align*}
where the last inequality is since $k_\epsilon = \lceil \log_{4/3} (\frac{\beta D^2}{2\epsilon}) \rceil$ by definition. 
Thus follows the claimed sample complexity of \ekgd\, in \cref{thm:epgd} and this concludes the proof.
\end{proof}

%% file: appendix2.tex
\section{Standard Results from Convex Optimization}
\label{app:std_res}	

The results covered in this section can be found in \cite{hazanbook,luenbook,boydbook,fletcherbook,nesterovbook,nocedalbook}.

\begin{defn}[Lipschitz Function]
Assume $\cD \subseteq \R^d$ be bounded decision space. Then any function $f: \cD \mapsto \R$ f is called $L$-Lipschitz over $\cD$ with respect to a norm $\norm{\cdot}$ if for all $\x, \y \in \cD$, we have: 
\[
\abs{f(\x) - f(\y)} \leq L\norm{\x - \y}.
\]
\end{defn}

\subsection{Useful properties for Convex Functions}

\begin{defn}[Convex Function]
\label{def:cvx}
Assume $\cD \subseteq \R^d$ be any convex and bounded decision space. Then any differential function $f: \cD \mapsto \R$ is called convex if for all $\x,\y \in \cD$,
\[
f(\x) - f(\y) \geq \nabla f(\y)^\top (\x - \y).
\]
\end{defn}

\begin{restatable}[First Order Optimality Condition~\citep{luenbook,boydbook}]{lem}{foo}
\label{lem:foo}
Assume $f: \cD \mapsto \R$ is a convex function and $\x^*$ be the minimizer of $f$. Then for any $\x \in \cD$, 
\[
\nabla f(\x^*)^\top (\y - \x^*) \geq 0
\]
\end{restatable}

\subsection{Useful properties for $\beta$-Smooth Convex Functions}

\begin{defn}[$\beta$-Smooth Convex Function]
\label{def:cvxb}
Assume $\cD \subseteq \R^d$ be any convex and bounded decision space. Then any differential and convex function $f: \cD \mapsto \R$ is also called $\beta$-smooth (any $\beta > 0$) if for all $\x,\y \in \cD$,
\[
f(\x) - f(\y) \leq \nabla f(\y)^\top (\x - \y) + \frac{\beta}{2}\norm{\x-\y}^2.
\]
\end{defn}

\begin{restatable}[Properties of $\beta$-smooth functions~\citep{hazanbook,bubeckbook}]{lem}{propb}
\label{lem:prop_b}
Suppose $f: \cD \mapsto \R$ is a $\beta$-smooth convex function. Then for all $\x,\y \in \R^d$, 
\begin{align*}
  f(\x) - f(\y) \leq \nabla f(\x)^\top(\x-\y) & - \frac{1}{2\beta}\norm{\nabla f(\x) - \nabla f(\y)}^2
  \\
  \norm{\nabla f(\x) - \nabla f(\y)} & \leq \beta \norm{ \x - \y}.
\end{align*}
Further if $\x^*$ is the minimizer of $f$ and $\x^* \in \text{Int}(\cD)$ (i.e. $\x^*$ belong to the interior of $f$'s domain $\cD$), then
\begin{align*}
~~& \norm{\nabla f(\x)}^2 \leq 2\beta (f(\x) - f(\x^*)
\\
~~& f(\x) - f(\x^*) \leq \frac{\beta}{2} \norm{ \x - \x^*}^2. 
\end{align*}	
\end{restatable}

\subsection{Useful properties for $\alpha$-Strongly Convex Functions}

\begin{defn}[$\alpha$-Strongly Convex Function]
\label{def:cvxa}
Assume $\cD \subseteq \R^d$ be any convex and bounded decision space. Then any differential and convex function $f: \cD \mapsto \R$ is also called $\alpha$-strongly convex (any $\alpha > 0$) if for all $\x,\y \in \cD$,
\[
f(\x) - f(\y) \geq \nabla f(\y)^\top (\x - \y) + \frac{\alpha}{2}\norm{\x-\y}^2.
\]
\end{defn}

\begin{restatable}[]{lem}{propa}
	\label{lem:prop_a}
	If $f: \cD \mapsto \R$ is an $\alpha$-strongly convex function, with $\x^*$ being the minimizer of $f$. Then for any $\x, \y, \z \in \cD$,
	\begin{align*}
	& \norm{\nabla f(\x)-\nabla f(\y)} \geq \alpha \|\x - \y\|
	\\
	& \norm{\nabla f(\z)} \geq \alpha \|\z - \x^*\|
	\\
	& \frac{\alpha}{2}\|\x^* - \x\|^2 \le f(\x) - f(\x^*).		
	\end{align*}
\end{restatable}	

\begin{proof}
		This simply follows by the properties of $\alpha$-strongly convex function. Note by definition of $\alpha$-strong convexity, for any $\x,\y \in \R$,
		\[
		f(\x) - f(\y) \ge \nabla f(\y)^\top (\x - \y) +  \frac{\alpha}{2}\|\x - \y\|^2. 
		\]
Similarly,
		\[
		f(\y) - f(\x) \ge \nabla f(\x)^\top (\y - \x) +  \frac{\alpha}{2}\|\x - \y\|^2. 
		\]
Adding we get:
\begin{align*}
(\nabla f(\y)-\nabla f(\x))^\top (\y - \x) \ge {\alpha}\|\x - \y\|^2
\end{align*}

Now applying Cauchy-Schwarz inequality to the left hand side of the above inequality yields the first result.

To get the second result, let us use $\y = \z$ and $\x = \x^*$ in the above inequality, which along with the first order optimality yields (\cref{lem:foo}):
\[
\nabla f(\z)^\top (\z - \x^*) \ge {\alpha}\|\z - \x^*\|^2
\]		 
The result now follows by again applying Cauchy-Schwarz inequality to the left hand side of the above inequality. 
Finally the last part of the proof simply follows setting $\y = \x^*$ and from the first order optimality condition (see \cref{lem:foo}).
\end{proof}

\subsection{Useful properties for $\alpha$-Strongly Convex and  $\beta$-Smooth Convex Functions $(\beta \geq \alpha)$}

\begin{restatable}[Properties of $\beta$-smooth and $\alpha$-strongly convex functions ~\citep{bubeckbook}]{lem}{propab}
	\label{lem:prop_ab}
	Suppose $f: \cD \mapsto \R$ is a $\beta$-smooth and $\alpha$-strongly convex function. Then for all $\x,\y \in \cD$
\begin{align*}
\big(\nabla f(\x) - \nabla f(\y)\big)^\top(\x-\y) \geq   \frac{\alpha\beta}{\alpha+\beta}\norm{\x-\y}^2
+ 
\frac{1}{\alpha+\beta}\norm{\nabla f(\x) - \nabla f(\y)}^2
\end{align*}	
\end{restatable}